\documentclass[12pt,a4paper,reqno]{amsart} 
\usepackage[margin=2.9cm]{geometry}
\usepackage{lmodern}
\usepackage[english]{babel}
\usepackage[utf8]{inputenc}
\usepackage[dvipsnames]{xcolor}
\usepackage[colorlinks=true,linkcolor=MidnightBlue, citecolor=RoyalBlue]{hyperref}
\usepackage{amsmath}
\usepackage{amsfonts}
\usepackage{amssymb}
\usepackage{amsthm}
\usepackage{bbm}
\numberwithin{equation}{subsection}
\newtheorem{thm}{Theorem}[subsection]
\newtheorem{cor}{Corollary}[subsection]
\newtheorem{lemma}{Lemma}[subsection]

\theoremstyle{remark}
\newtheorem*{remark}{Remark}
\theoremstyle{definition}

\newcommand{\bz}{\mathbb{Z}}
\newcommand{\br}{\mathbb{R}}
\newcommand{\bn}{\mathbb{N}}

\newcommand{\fC}{\mathfrak{C}}

\newcommand{\cC}{\mathcal{C}}
\newcommand{\cD}{\mathcal{D}}

\newcommand{\cH}{\mathcal{H}}\newcommand{\cI}{\mathcal{I}}

\newcommand{\cM}{\mathcal{M}}
\newcommand{\cN}{\mathcal{N}}
\newcommand{\sumstar}{\sideset{}{^*}\sum}

\title{Second moment of degree three $L$-functions}
\author{Sampurna Pal}
\address{Theoretical Statistics and Mathematics Unit, Indian Statistical Institute, Kolkata, WB, India}
\email{sampu.andul@gmail.com}

\begin{document}
\begin{abstract}Let $F$ be a Hecke-Maa\ss\  cusp form for $\mathrm{SL}(3,\mathbb{Z})$. We obtain the first non-trivial upper bound of the second moment of $L(F,s)$ in $t$-aspect:
		$$	\int_{T}^{2T}|L(F,1/2+it)|^2 dt\ll_{F,\varepsilon} T^{3/2-3/32+\varepsilon}.$$
			Immediate corollaries include improvements over the existing results on the subconvexity bound for self-dual $\mathrm{GL}(3)$ $L$-functions in the $t$-aspect and for self-dual $\mathrm{GL}(3)\times \mathrm{GL}(2)$ $L$-functions in the $\mathrm{GL}(2)$ spectral aspect, the error term in the Rankin-Selberg problem, and the zero density estimate for $\mathrm{GL}(3)$ $L$-functions.
            \end{abstract}

\maketitle{}
	\setcounter{tocdepth}{1}

	\section{Introduction}
 Let $F$ be a Hecke-Maa\ss \  cusp form for $\mathrm{SL}(3,\mathbb{Z})$ of type $\nu=(\nu_1,\nu_2)\in \mathbb{C}^2$ with normalized Fourier coefficients $A(m,n)$ (so that $A(1,1)=1$) and let $L(F,s)$ be its Godement-Jacquet $L$-function defined by$$L(F,s)=\sum_{n=1}^{\infty}A(1,n)n^{-s}\ \ \ \text{for } \Re(s)>1.$$ Then, we denote its  second moment in the $t$-aspect at the critical line as 
	\begin{equation}\label{FirstDefnofSecondMoment}
		M_F(T)=\int\limits_{T}^{2T}\left|L\left(F,1/2+it\right)\right|^2 dt.
	\end{equation}
 The generalized Lindel\"of hypothesis conjectures that we may bound it by $M_F(T)\ll T^{1+\varepsilon}$. On the other hand, the approximate functional equation implies a trivial upper bound of $M_F(T)\ll T^{3/2+\varepsilon}$. But any improvement over the exponent $\frac{3}{2}$ has remained an open problem so far. In this paper, we will address it by establishing the first non-trivial upper bound of the second moment $M_F(T)$ by proving the following.
	\begin{thm}\label{maintheorem}Let $F$ be a Hecke-Maa\ss \  cusp form for $\mathrm{SL}(3,\mathbb{Z})$. Then we have 
		\begin{equation}\label{maintheoremeqn}
			\int\limits_{T}^{2T}\left|L\left(F,1/2+it\right)\right|^2 dt\ll_{F,\varepsilon} T^{\frac{3}{2}-\frac{3}{32}+\varepsilon}.
		\end{equation}
	\end{thm}

	In Analytic number theory, understanding the integral moments of $L$-functions is of fundamental interest. 

	One of the very first results regarding the integral moments is due to Hardy and Littlewood, who in \cite{HL} derived an asymptotic expression of the second moment of the Riemann zeta function: $$\int_0^T\left|\zeta\left(1/2+it\right)\right|^2dt\sim\  T\log T.$$
	This asymptotic expression was improved by Atkinson \cite{At}. Then Ingham \cite{In} obtained an asymptotic formula for the fourth moment of the Riemann zeta function: $$\int_0^T\left|\zeta\left(1/2+it\right)\right|^4dt\sim\frac{1}{2\pi^2}\ T\log^4 T.$$ It was subsequently improved by Heath-Brown \cite{Hb2} and generalized to weighted moments by Motohashi (see \cite{Mo}). Other than these, no asymptotic expression is known for the higher moments ($k\geq 3$) of the Riemann zeta function. But it is conjectured that there exists some constant $C_k$ such that $\int_0^T\left|\zeta\left(1/2+it\right)\right|^{2k}dt\sim C_k\  T \log^{k^2}T$
 ; see \cite{CGh}, \cite{CGo}, \cite{KS},  \cite{DGH}, \cite{CFKRS}. 

 Though we do not have an asymptotic expression of the sixth and higher moments of the  Riemann zeta function, we do know of an upper bound for the twelfth moment of the Riemann zeta function due to  Heath-Brown \cite{Hb1}: $\int_0^T\left|\zeta\left(1/2+it\right)\right|^{12}dt\ll T^{2+\varepsilon}.$ From this result, with a simple application of Cauchy-Schwarz inequality, one can get non-trivial upper bounds of sixth, eighth and tenth moments. In particular, this implies a non-trivial upper bound of the sixth moment of the Riemann zeta function  $$\int_0^T\left|\zeta\left(1/2+it\right)\right|^{6}dt\ll T^{5/4+\varepsilon},$$
 which is an improvement over the trivial bound $T^{3/2+\varepsilon}$. The sixth moment of the Riemann zeta function is related to the second moment of the $L$-function associated with a minimal parabolic Eisenstein series for $\mathrm{SL}(3,\mathbb{Z})$. Hence, our main result (Theorem \ref{maintheorem}), i.e., the second moment of the $\mathrm{GL}(3)$ $L$-functions, can be thought of as a cuspidal analog of this problem.
  
    For degree two $L$-functions, Good \cite{Go} derived an asymptotic expansion for the second moment of the $L$-functions associated with holomorphic cusp forms. For a holomorphic cusp form $f(z):=\sum_{n=1}^{\infty}a_ne(nz)$ of weight $k$ and $L$-function $L(f,s)=\sum_{n= 1}^{\infty}a_nn^{-s}$ for $\Re(s)>\frac{k+1}{2}$, he proved 
	\begin{equation}\label{resultofgood}
		\int_0^T\left|L\left(f,k/2+it\right)\right|^2dt= 2c_{-}T\left(\log \left(T/2\pi e\right)+c_0\right)+O((T\log T)^{\frac{2}{3}}),\end{equation}
	for some constants $c_{-},c_0$ computed explicitly. Here, the power saving error term was important as it resulted in a subconvex bound of $L(f,s)$. But any asymptotic estimates of higher moments of degree two $L$-functions are still unknown.

    For degree three $L$-functions, though there has not been any improvement over the trivial bound of  $M_F(T)$ (\ref{FirstDefnofSecondMoment}) until this work, different variations of this expression have been extensively studied (cf. \cite{Li2}, \cite{Yo}). In a remarkable recent work, Aggarwal, Leung, and Munshi \cite{ALM} derived an upper bound for the second moment in a short interval: $[T-M, T+M]$ for $T^{1/2}<M<T^{1-\varepsilon}$. 
	  As the result is intended for the short moment, it fails to give any non-trivial upper bound for the full second moment ($M=T$).

  To obtain the trivial bound of $M_F(T)\ll T^{3/2+\varepsilon}$ (\ref{FirstDefnofSecondMoment}), we start with the application of the approximate functional equation (Lemma \ref{approximatefunctionaleqn}) by which we can truncate the sum over $n$ in the integrand up to $N=T^{3/2+\varepsilon}$. After opening up the absolute square and evaluating the oscillatory integral, we get 
	\begin{equation}\label{shiftedconvolutionsumfirst}
		M_F(T)\ll T^{-1/2}\left|\sum_{h\sim H}\sum_{n\sim N}A(1,n)A(1,n+h)\right|,\end{equation}
	for $H=T^{1/2+\varepsilon}$. The double sum can be bounded above by $NH$ by Cauchy-Schwarz inequality and the Ramanujan bound on average (\ref{ramanujanonaverage}). Thus, we get the trivial upper bound $M_F(T)\ll T^{3/2+\varepsilon}$.

 Thus, a natural way to obtain a non-trivial estimate of (\ref{FirstDefnofSecondMoment}) would have been obtaining a non-trivial upper bound of the shifted convolution sum (\ref{shiftedconvolutionsumfirst}), specifically for $H\sim N^{1/3}$ ($N\sim T^{3/2+\varepsilon}$). 
 Instead, we split the range of the $t$-integral $[T,2T]$ into shorter ranges of length $\xi\sim X$ and then evaluate the $t$-integral (see \ref{2.11}). At this stage, we also introduce an extra averaging of $\xi$ (new length of the $t$ integrals) in the range $[X,2X]$, i.e., we introduce an integral of the form $\frac{1}{X}\int_{\xi\sim X}$ (see \ref{2.4}). This is one of the crucial inputs of our paper. Here, $X$ is a new parameter that will be determined optimally at the end. This $\xi$ integral will be essential for obtaining the non-trivial bound, as it will be apparent at the end. 
 
 Now, we apply the delta method of Duke, Friedlander, and Iwaniec along with Munshi's conductor lowering trick to separate the oscillations. Then, we dualize the sums with the Voronoi summation formula and the Poisson summation formula. Then, similar to \cite{ALM}, we apply the duality principle of the large sieve to interchange the order of the summations and eliminate the Fourier coefficients $A(m,n)$ of $F$. After that, we apply the Poisson summation formula twice, and each time, we get a character sum and an oscillatory integral. We then carefully evaluate these character sums and oscillatory integrals to recover the trivial bound of $M_F(T)$. At this point, any extra cancellation would give us a non-trivial bound. 
 
 And we do have an oscillatory term left, reminiscent of the extra oscillatory term introduced at the beginning. We also have the extra integral we introduced at the beginning, precisely the integral over $\xi$ in the $[X,2X]$ range. Thus, we extract an extra saving out of that oscillatory integral, which is enough to give us the non-trivial bound in Theorem \ref{maintheorem}.  For a more detailed sketch of the proof, see Section \ref{S2}.

\textit{Remark added August $7, 2024$}: After the appearance of the first version of the present paper,   Dasgupta, Leung, and Young \cite{DLY} have further improved the exponent of Theorem \ref{maintheorem} and consequently the exponents in the following corollaries. 

	\subsection{Applications}
	Theorem \ref{maintheorem} also has some important consequences in the subconvexity problem, the Rankin-Selberg problem, and the zero density estimates.  
	\subsubsection{Subconvex bound of $\mathrm{GL}(3)$ $L$-functions:}
	For any degree $d$ $L$-function $L(f,s)$, the convexity principle provides the convexity bound in the $t$-aspect, i.e. $L(f,1/2+it)\ll_{f}(1+|t|)^{d/4+\varepsilon}$ and any improvement of this upper bound is called a subconvex bound in the $t$-aspect.

    For degree three $L$-functions, the first subconvex bound in $t$-aspect was established by Li \cite{Li2} for self-dual forms for $\mathrm{SL}(3,\mathbb{Z})$. It was extended to any general forms for $\mathrm{SL}(3,\mathbb{Z})$ by Munshi \cite{Mu1}.   In a recent pre-print, Aggarwal, Leung, and Munshi \cite{ALM} improved the subconvex bound of \cite{Mu1} by bounding a short second moment.	We also note that in a recent pre-print, Nelson \cite{Ne} has obtained the $t$-aspect subconvexity bound for standard $L$-function of any degree $d$. 
	
	Theorem \ref{maintheorem} also implies a $t$-aspect subconvex bound of degree three $L$-functions:\\
	\begin{cor}\label{cor1}Let $F$ be a Hecke-Maa\ss \  cusp form for $\mathrm{SL}(3,\mathbb{Z})$. Then we have
		\begin{equation}\label{subconvexitybound}
			L(F,1/2+it)\ll_{F,\varepsilon} (1+|t|)^{3/4-3/64+\varepsilon}.
		\end{equation}
	\end{cor}
	
	\begin{remark} The bound we obtain is weaker than the bound obtained in  \cite{Mu1}.\end{remark}
	\subsubsection{Subconvex bound of self-dual $\mathrm{GL}(3)$ $L$-functions:}
	Let $F$ be a self-dual Hecke-Maass cusp form for $\mathrm{SL}(3,\mathbb{Z})$ and let $u_j$ be an orthonormal basis of even Hecke-Maass cusp form for $\mathrm{SL}(2,\mathbb{Z})$ of Laplacian eigenvalue $\frac{1}{4}+t_j^2$. With these assumptions Li, in her pioneering work \cite{Li2}, proved the subconvex bounds $$L(F,1/2+it)\ll_{F,\varepsilon}(1+|t|)^{11/16+\varepsilon},\ L(F\otimes u_j,1/2)\ll_{F,\varepsilon}(1+|t_j|)^{11/8+\varepsilon}.$$ These bounds were later improved by McKee, Sun, and Ye \cite{MSY} and by Nunes \cite{Nu}. Recently, Lin, Nunes, and Qi \cite{LNQ} reached the limit of the moment method of Li. Let $f_j$ be a Hecke-Maass cusp form (not necessarily even) for $\mathrm{SL}(2,\mathbb{Z})$ of Laplacian eigenvalue $\frac{1}{4}+t_j^2$. Then they proved  
	$$L(F,1/2+it)\ll_{F,\varepsilon}(1+|t|)^{3/5+\varepsilon}, L(F\otimes f_j,1/2)\ll_{F,\varepsilon}(1+|t_j|)^{6/5+\varepsilon}.$$ 
	Their bounds rely on an upper bound of the second moment of $\mathrm{GL}(3)$ $L$-functions. Thus, our non-trivial upper bound of the second moment of $\mathrm{GL}(3)$ $L$-functions (Theorem \ref{maintheorem}) leads to the following improvements of the above bounds : 
	\begin{cor}\label{cor2} Let $F$ be a self-dual Hecke-Maa\ss \  cusp form for $\mathrm{SL}(3,\mathbb{Z})$ and let $f_j$ be a Hecke-Maa\ss\  cusp form for $\mathrm{SL}(2,\mathbb{Z})$ of Laplacian eigenvalue $\frac{1}{4}+t_j^2$. Then we have
		\begin{equation*}
			L(F,1/2+it)\ll_{F,\varepsilon}(1+|t|)^{45/77+\varepsilon}, L(F\otimes f_j,1/2)\ll_{F,\varepsilon}(1+|t_j|)^{90/77+\varepsilon}.
		\end{equation*}
	\end{cor} 
	\subsubsection{Rankin-Selberg Problem }
	Let $f$ be a normalized holomorphic Hecke cusp form or a Hecke-Maa\ss\  cusp form for $\mathrm{SL}(2,\mathbb{Z})$ and let $\lambda_f(n)$ be its $n$-th Hecke eigenvalue. The goal of the Rankin-Selberg problem is to bound the error term for the second moment of $\lambda_f(n)$, i.e. 
	$$\Delta(X)=\sum_{n\leq X}\lambda_f(n)^2-c X, $$ 
    where $c=L(\text{Sym}^2f,1)/\zeta(2)$. Rankin \cite{Ra} and Selberg \cite{Se} established the longstanding upper bound of $\Delta(X)\ll X^{3/5}$. Huang lowered the upper bound in \cite{Hu} establishing $\Delta(X)\ll X^{3/5-1/560}$. Theorem \ref{maintheorem} implies an improvement of this bound: 
	\begin{cor}\label{cor3} Let $f$ be a holomorphic Hecke cusp form or a Hecke-Maa\ss\  cusp form for $\mathrm{SL}(2,\mathbb{Z})$ and let $\lambda_f(n)$ be its $n$-th Hecke eigenvalue. Then we have 
		$$\Delta(X)\ll_{f,\varepsilon}  X^{45/77+\varepsilon}=X^{3/5-6/385+\varepsilon}.$$
	\end{cor} 
 
	\subsubsection{Zero density estimate for $\mathrm{GL}(3)$ $L$-functions} Let $L(f,s)$ be any $L$-function associated to an automorphic form $f$. Now by $N(\sigma,T)$ we denote the number of zeros $\rho=\beta+i\gamma$ of $L(f,s)$ in the region $\beta\geq \sigma$ and $|\gamma|\leq T$ where $\frac{1}{2}\leq \sigma\leq 1$. In the absence of any asymptotic expression, upper bounds of $N(\sigma,T)$ of the form 
	$$N(\sigma,T)\ll T^{A(\sigma)(1-\sigma)+\varepsilon}$$
	are usually referred to as zero-density estimates in the literature. 
    Mukhopadhyay and Srinivas \cite{MS} obtained zero density estimates for $L$-functions of arbitrary degree depending on the upper bound of the second moment for that $L$-function: 
    $$\int\limits_0^T|L(f,1/2+it)|^2 dt\ll T^{\alpha+\varepsilon}\implies N(\sigma,T)\ll \begin{cases}
    	T^{2\alpha(1-\sigma)+\varepsilon} &\text{for }\frac{1}{\alpha}\leq\sigma\leq 1 \\
    	T^{\frac{2}{\sigma}(1-\sigma)+\varepsilon}& \text{for }\frac{2}{3}\leq \sigma\leq \frac{1}{\alpha}
    \end{cases}.$$	
  In \cite{YZ}, Ye and Zhang obtained zero density estimates for $L$-functions of arbitrary degree depending on the upper bound of any $2k$th moment. Considering only the second moment bound, we may restate their result as 
  $$\int\limits_0^T|L(f,1/2+it)|^2 dt\ll T^{\alpha+\varepsilon}\implies  N(\sigma,T)\ll T^{\frac{2(1+\alpha)(1-\sigma)}{3-2\sigma}+\varepsilon}\text{ for }\frac{1}{2}\leq \sigma\leq 1.$$  Then if $L(F,s)$ be a $\mathrm{GL}(3)$ $L$-function and as both of these results depends on the upper bound of $t$-aspect moments of $L(F,s)$, our non-trivial bound of the second moment of $\mathrm{GL}(3)$ $L$-functions (\ref{maintheoremeqn}) implies the following improvement of the zero density estimates for $\mathrm{GL}(3)$ $L$-functions. 
	\begin{cor}\label{cor4}Let $L(F,s)$ Hecke-Maa\ss \  cusp form for $\mathrm{SL}(3,\mathbb{Z})$ and let $N(\sigma,T)$ denotes the number of zeros $\rho=\beta+i\gamma$ of $L(F,s)$ in the region $\beta\geq \sigma$ and $|\gamma|\leq T$ where $\frac{1}{2}\leq \sigma\leq 1$. Then we have
		$$N(\sigma,T)\ll \begin{cases}
			T^{(3-3/16)(1-\sigma)+\varepsilon}, &\text{for }\frac{32}{45}\leq \sigma\leq 1\\
			T^{\frac{(77/16)(1-\sigma)}{(3-2\sigma)}+\varepsilon}, &\text{for }\frac{1}{2}\leq \sigma \leq 1 
		\end{cases}.$$
	\end{cor}
	\subsection{Notations} We follow the popular notation of $e(x):=e^{2\pi i x}$ and $\varepsilon$ would denote an arbitrarily small positive real number. $f=O(A)$ and $f\ll A$ would denote $|f|\leq C AT^{\varepsilon}$ for some constant $C>0$ depending only on the form $F$ and $\varepsilon$. For some $n>0$, $x\sim n$ would imply $\ c_1 n< x< c_2 n$ where $0<c_1<c_2$ and $x\in\mathbb{R}$. In general, we will reserve this notation for the range of sum and integrals. $f\asymp g$ would imply $T^{-\varepsilon}<|f/g|<T^{\varepsilon}$ when $g$ is not zero. We will reserve the letter $V,W$ for compactly supported smooth function. They will additionally satisfy $x^{j}V^{(j)}(x)\ll 1$ for $j\in \mathbb{N}$, unless specified otherwise. The definition of these smooth functions $V,W$  will vary from place to place. For brevity, often we will omit the underlying smooth function of the form $V\left(\frac{x}{M}\right)$ in a sum or integral by denoting $\sum_{x\sim M}$ or $\int_{x\sim M}$. Any contribution is called negligible if the term is $O(T^{-A})$ for some fixed large $A$, say $A=1999$. We will frequently use the familiar notations $e(n)=e^{2\pi i n}$, $e_q(n)=e^{2\pi i n/q}$, $S(a,b,q)=\sumstar_{n\bmod q}e_q(a\bar n+bn)$ and $\mathfrak c_q(a)=\sum_{n\bmod q}e_q(an)$.   
	
	\section{Sketch of the proof}\label{S2}
	To prove Theorem \ref{maintheorem}, we start with the integral of $M_F(T)$ (\ref{FirstDefnofSecondMoment}) though we take the range to be $[2T,3T]$. By approximate functional equation, we can truncate the length of the Dirichlet series of $L(F,1/2+it)$ to $n\sim  N$ for $N\ll T^{3/2+\varepsilon}$. For the sketch, we will only consider $N=T^{3/2}$. So, we start with (see \ref{2.1})
    $M(T)=\int_{2T}^{3T}\left|\sum_{n\sim N}A(n)n^{-it}\right|^2dt.$ In Section \ref{S5}, we split the $t$ integral into smaller integrals of length $\xi$ and then take an average over $\xi$ (see \ref{2.4}), in the form of an integral on the range $[X,2X]$, where $X$ would be optimally chosen at the end.
	$$\int\limits_{2T}^{3T}|\cdots|^2dt\rightsquigarrow\sum_{\frac{T}{\xi}\leq r< \frac{2T}{\xi}}\int\limits_{T+r\xi}^{T+(r+1)\xi}|\cdots|^2dt\ll \sum_{\frac{T}{2X}\leq r<\frac{2T}{X}}\frac{1}{X}\int_{\xi\sim X}\int\limits_{T+r\xi}^{T+(r+1)\xi}|\cdots|^2dt$$

    We now open up the absolute square to get $\sum_h\sum_nA(1,n)\overline{A(1,n+h)}\int_{t}\cdots$. Then we evaluate the $t$-integral (see Lemma \ref{tintegrallemma}), which restricts the range of $h$ to be $h\ll H=NT^{\varepsilon}/X$, as the $t$-integral would be negligible otherwise. We essentially have $M_F(T)\ll N^{-1}\sum_{r\sim T/X}\mathfrak{M} $, where 
	$$\mathfrak{M}=\int\limits_{\xi\sim X}\sum_{h\ll H}\sum_{n\sim N }A(1,n)\overline{A(1,n+h)}e\bigg(\frac{(T+r\xi)h}{2\pi n}\bigg).$$
	Comparing it with (\ref{shiftedconvolutionsumfirst}), we note that the fragmentation of the integral has increased the size of $h$, introduced an extra oscillatory factor, and along with it, there is an extra integral over $\xi$. In the end, these two elements would give us the extra saving required for the non-trivial bound of Theorem \ref{maintheorem}.
	
	Now, we will separate the oscillations of $\mathfrak{M}$ with the circle method. In Section \ref{S6}, we would use the delta method of Duke, Friedlander, and Iwaniec (Lemma \ref{delta}) with modulus $Q=\sqrt{N/(T/X)}=\sqrt{X}T^{1/4}$. 
	\begin{equation}\label{2.17}
		\begin{split}
			\mathfrak{M}=&\int_{x\sim 1}\frac{1}{Q}\sum_{q\sim Q}\frac{1}{q}\ \sideset{}{^*}\sum_{a\bmod q}\int_{\xi\sim X}\sum_{h\ll H}e_q(-ah)\times\bigg[\sum_{n\sim N}A(1,n)e_q(-an)e\bigg(\frac{(T+r\xi)}{2\pi}\times \frac{h}{n}\bigg)\bigg]\\
			&\times\bigg[\sum_{m}A(1,m)e_q(am)e\bigg(\frac{(m-n-h)x}{qQ}\bigg)W\bigg(\frac{m-n-h}{Q^2}\bigg)\bigg].
		\end{split}
	\end{equation}
	Here, $W(x)$ is an even smooth function supported in $[-1,1]$ and $W(0)=1$. If we trivially evaluate $\mathfrak{M}$, we get $\mathfrak{M}\ll X\cdot H\cdot N\cdot Q^2=Q^2N^2\implies M_F(T)\ll T^3$. So in $\mathfrak{M}$, we need to save  $T^{3/2}$ and a little more to get a non-trivial bound; that is we need to prove $\mathfrak{M}\ll \frac{Q^2N^2}{T^{3/2+\delta}}$ for some $\delta>0$.
	
	To treat these $m$ and $n$-sums above, we employ the Voronoi type summation formula for $\mathrm{SL}(3,\mathbb{Z})$ (Lemma \ref{voronoi}). In Section \ref{S7}, we apply the Voronoi summation formula first on the $m$-sum and then on the $n$-sum. Though we have $A(m_2,m_1)$ in the RHS of the Voronoi summation formula along with the condition $m_1|q$, for brevity, we would only consider the ``generic case" $m_1=1$ in this sketch and write $m_2$ as $m$. In the $m$-sum, we encounter an oscillatory integral, and by repeated integration by parts, it restricts the range of $m\ll (T/Q)^{3+\varepsilon}$ (see Lemma \ref{L7.3}). Then, by Taylor series approximation of the phase function, the $m$ sum transforms into 
	$$\frac{Q^2}{N}\sum_{m\ll(T/Q)^3}A(m)S(\bar a, m;q)e\bigg(\frac{3m^{1/3}n^{1/3}}{q}\bigg)
	.$$ Thus in the $m$-sum we save $$\frac{\text{initial bound}}{\text{final bound}}=\frac{Q^2}{\frac{Q^3}{qN}\cdot\frac{T^3}{Q^3}\cdot q^{1/2}}\sim \frac{X^{5/4}}{T^{7/8}}.$$ We take the above oscillatory term into the $n$-sum and apply the Voronoi summation formula. There, we also face an oscillatory integral $I_2$ (see \ref{Iintegralnsum}) of trivial size $T$. So, the $n$ sum transforms to 
	\begin{equation*}\begin{split}\frac{1}{T}\sum_{n\ll(T/Q)^3}A(n)S(-\bar a,n;q)I_2.\end{split}
	\end{equation*}Thus in the $n$-sum we save 
	$$\frac{\text{initial bound}}{\text{final bound}}=\frac{N}{\frac{1}{T}\cdot \frac{T^3}{Q^3}\cdot q^{1/2}\cdot T}\sim \frac{X^{5/4}}{T^{7/8}}.$$ Finally, we put the oscillatory integral $I_2$ in the $h$-sum and apply the Poisson summation formula on $\sum_{h\ll H}e_q(ah)I_2(h).$ Thus, we get a double integral $I_3$ (see \ref{I3hsum}). This restrict the dual variable $h$ in the range $h\sim \frac{qT}{N}\sim \frac{\sqrt{X}}{T^{1/4}}$. After careful evaluation of the double integral, the $h$ sum transforms into (see Lemma \ref{Shfinallemma})
    $$T^{3/2}\sum_{\substack{h\sim \frac{\sqrt{X}}{T^{1/4}}\\h\equiv a\bmod q}} e\bigg(\frac{3(T+r\xi)^{1/3}(m^{1/3}-n^{1/3} )}{(2\pi)^{1/3}q^{2/3}h^{1/3}}\bigg).$$
	In the $h$-sum we save 
	$$\frac{\text{initial bound}}{\text{final bound}}=\frac{HT}{T^{3/2}\cdot \frac{1}{q}\cdot \frac{\sqrt{X}}{T^{1/4}}}\sim \frac{T^{3/2}}{X}.$$ 
	Hence we evaluate the total saving on $\mathfrak{M}$ to be $X^{3/2}/T^{1/4}$ and by Cauchy's inequality and the symmetry of the $n$ and the $m$ sum, $\mathfrak{M}$ can be bounded by 
	\begin{equation*}\begin{split}\mathfrak{M}\ll&\frac{T^{\varepsilon}}{T}\int_{\xi\sim X}\sum_{q\sim Q}\sum_{h\sim\sqrt{X}/T^{1/4}}\times\bigg|{\sum_{m\sim(T/Q)^3}}A(m)S(\bar h,m;q)e\bigg(\frac{3(T+r\xi)^{1/3} m^{1/3}}{(2\pi)^{1/3}q^{2/3}h^{1/3}}\bigg)\bigg|^2.
	\end{split}\end{equation*}
	
	In Section \ref{S8}, similar to \cite{ALM}, we use the duality principle of large sieve to interchange the order of sums in the expression of $\mathfrak{M}$. The Ramanujan bound on average (\ref{ramanujanonaverage}):  $\sum_{m_1^2m_2\leq X}|A(m_1,m_2)|^2\ll X^{1+\varepsilon}$ would imply that $\mathfrak{M}$ is bounded above by $$\mathfrak{M}\ll\frac{T^{+\varepsilon}}{T}\sum_{m\sim(T/Q)^3}|A(m)|^2\Delta\ll \frac{T^{2+\varepsilon}}{Q^3}\Delta,$$ where
	\begin{equation*}
		\begin{split}
			\Delta=\underset{||\alpha||_2=1}{\sup\limits_{\alpha}}\sum_{m\sim(T/Q)^3}\bigg|\int_{\xi\sim X}\sum_{q\sim Q}\sum_h\alpha(\xi,q,h)S(\bar h,m;q)\times e\bigg(\frac{3(T+r\xi)^{1/3} m^{1/3}}{(2\pi)^{1/3}q^{2/3}h^{1/3}}\bigg)d\xi\bigg|^2.
		\end{split}
	\end{equation*}
	Trivially we have $\Delta \ll T^{5/2}X$. To bound $\Delta$, we first open up the absolute square and apply the Poisson summation formula to the $m$-sum. We note that we denote the two copies of the variables $(\xi,q,h)$ as $(\xi_1,q_1,h_1)$ and $(\xi_2,q_2,h_2)$.  This produces a character sum $\mathfrak{C}$ and an oscillatory integral $\mathfrak{J}$ (see  \ref{CharacterSum}).
	$$	\sum_{m\sim (T/Q)^3}S(\bar h_1,m;q_1)S(\bar h_2,m;q_2)e(\cdots)=\frac{T^3}{Q^3q_1q_2}\sum_{m\in\mathbb{Z}}\mathfrak{C}\ \mathfrak{J},$$
	where
	\begin{equation*}
		\begin{split}
			\mathfrak{C}&=\sum_{\beta\bmod q_1q_2}S(\bar h_1,\beta;q_1)S(\bar h_2,\beta;q_2)e_{q_1q_2}(m\beta),~~
			\mathfrak{J}=\int_{y\sim 1}e(\cdots \ y^{1/3})e\bigg(-\frac{mT^3y}{Q^3q_1q_2}\bigg)dy.
		\end{split}
	\end{equation*} At this point, we will consider two cases of this expression: ``Diagonal" ($m=0$) and ``Off-diagonal" ($m\neq 0$). We will analyze them separately. 
	
	In Section \ref{S9}, we consider the diagonal case ($m=0$). In this case, the character sum turns out to be a Ramanujan sum along with the condition $q_1=q_2$:
	$\mathfrak{C}=q_1^2\mathfrak{c}_{q_1}(\bar{h}_1-\bar{h}_2)\delta_{q_1=q_2}.$
	We analyze the oscillatory integral with repeated integration by parts (Lemma \ref{repeatedint}) and observe that it is arbitrarily small unless $\xi_1-\xi_2\ll\frac{X^2}{T}$. By careful analysis of these two conditions, we see that we save $Q\frac{\sqrt{X}}{T^{1/4}}\cdot \frac{T}{X}\sim T$ in the diagonal case (see \ref{Deltam=0}). Before the application of the duality principle, we have saved $\frac{X^{3/2}}{T^{1/4}}$. So, this saving is sufficient as long as $\frac{X^{3/2}}{T^{1/4}}T\gg T^{3/2}\iff X\gg T^{1/2}$.
	
	In the off-diagonal case, both the character sum and the oscillatory integrals are more complicated. In Section \ref{S10}, we first analyze the oscillatory integral $\mathfrak{J}$ with stationary phase analysis (see Lemma \ref{Jintegrallemma}). For $m\ll \frac{X^{3/2}}{T^{3/4}}$ we have
	$$\mathfrak{J}\asymp \frac{1}{\sqrt{T/X}}e\left(2\sqrt{\frac{q_1q_2}{m}}\left(\Xi_2-\Xi_1\right)^{3/2}\right),\text{ where },~~\Xi_k=\left(\frac{(T+r\xi_k)}{2\pi q_k^2h_k}\right)^{1/3}\sim \frac{T^{1/4}}{\sqrt{X}},$$
   for $k=1,2$. We also note that the oscillation is of the size $T/X$.
    
    Next, we simplify the character sum $\mathfrak{C}$. For simplicity if we assume that $(q_1,q_2)=1$, by Chinese remainder theorem and reciprocity we can transform $\mathfrak{C}$ into 
	$$\mathfrak{C}\asymp q_1q_2\ e\left(\frac{q_2\bar q_1}{mh_1}+\frac{q_1\bar q_2}{mh_2}\right).$$ In reality there are additional terms depending on the g.c.d. of $q_1$ and $q_2$ and it requires careful evaluation (see \ref{CharsumSimplifiedA1A2}). So by the Poisson summation formula on the $m$-sum (in off-diagonal case ) we save $$\frac{\text{initial bound}}{\text{final bound}}=\frac{\frac{T^3}{Q^3}\cdot Q^{1/2}Q^{1/2}}{\frac{T^3}{Q^3}\cdot \frac{X^{3/2}}{T^{3/4}}\cdot \frac{\sqrt{X}}{\sqrt{T}}}\sim \frac{T^{3/2}}{X^{3/2}}.$$
	So, if we look at the expression, we essentially have 
	$$\sum_{m}\int_{\xi_1}\int_{\xi_2}\sum_{h_1}\sum_{h_2}\sum_{q_1}\sum_{q_2}\alpha_1\bar\alpha_2e\left(\frac{q_2\bar q_1}{mh_1}+\frac{q_1\bar q_2}{mh_2}\right)e(\cdots).$$
	We then apply the Cauchy-Schwarz inequality on the $(\xi_1,q_1,h_1)$-sum to get rid of $\alpha_1$ and open up the resulting absolute square and denote the resulting two copies of $(\xi_2,q_2,h_2)$ variables as $(\xi_2,q_2,h_2)$ and $(\xi_2',q_2',h_2')$. We now apply the Poisson summation formula on the $q_1$ sum. This produces another character sum $\mathcal{C}$ and oscillatory integral (see  \ref{afterpoissonofq1}). We take the $\xi_1$ integral along with the oscillatory integral and denote the two variable oscillatory integral as $\mathcal{I}$.
	\begin{equation*}\label{key}
		\begin{split}
			\int_{\xi_1}\sum_{q_1\sim Q}e\left(\frac{q_2\bar q_1}{mh_1}+\frac{q_1\bar q_2}{mh_2}-\frac{q_2'\bar q_1}{mh_1}-\frac{q_1\bar q_2'}{mh_2'}\right)e(\cdots)=\frac{Q}{mh_1h_2h_2'}\sum_{q_1\in \mathbb{Z}}\mathcal{C}\ \mathcal{I}.
		\end{split}
	\end{equation*}
	If we continue with our assumption $(q_1,q_2q_2')=1$, then by the Chinese remainder theorem character sum $\mathcal{C}$ appears to be a product of a  Kloosterman sum and a congruence relation: 
	$$\mathcal{C}\asymp S(\cdots,\cdots;mh_1)\cdot h_2h_2'\cdot \delta (\cdots\bmod h_2h_2').$$
	So we may bound $\mathcal{C}$ by $\sqrt{mh_1}$ on average. But in reality, our analysis of $\mathcal{C}$ depends on the common divisors of $q_1,q_2,q_2'$. On the other hand, we analyze the double integral $\mathcal{I}$ by multi-variable stationary phase analysis (see Lemma \ref{AnalysisofI}) and save $\frac{T}{X}$.  
    
    So, due to the Poisson summation formula on $q_1$ we save $$\frac{\text{initial bound}}{\text{final bound}}=\frac{QX}{\frac{Q}{\frac{X^{3}}{T^{3/2}}}\frac{X^{3/2}}{T^{3/4}}\cdot \frac{X}{\sqrt{T}}\frac{X^2}{T}}\sim \frac{T^{3/4}}{X^{1/2}}.$$ After the large sieve, the total saving in the off-diagonal case is $\frac{T^{3/2}}{X^{3/2}}\cdot\frac{T^{3/8}}{X^{1/4}}\sim \frac{T^{15/8}}{X^{7/4}}$. With the optimal choice of $X=T^{1/2}$, this saving equals the diagonal saving. Hence, we are at the threshold, i.e., we recover the trivial bound of $M_F(T)$. So, any additional saving in the off-diagonal will lead to a non-trivial bound of $M_F(T)$.   
	
	We now have an oscillatory term, which is the exponential term arising from the stationary phase analysis of $\mathcal{I}$ and its oscillation is of the size $T/X$. We apply Cauchy's inequality, keeping everything but the $\int_{\xi_2'}\alpha_2'e(\cdots )$ inside. So, we get 
    $$\int_{\xi_2'}\alpha_2'\int_{\xi_2''}\alpha_2''\int _{\xi_2}e(\cdots ).$$
    Then, by repeated integration by parts (Lemma \ref{repeatedint}) on the $\xi_2$ integral, we restrict the range of $|\xi_2'-\xi_2''|\ll \frac{X}{T/X}$ (see \ref{Xi2'X-2''range}). Thus, we save $T/X$. As we have applied Cauchy's inequality twice, the effective saving in $M_F(T)$ is $\frac{T^{1/4}}{X^{1/4}}$. 
    
    Hence, the total saving in the off-diagonal (after duality ) is $\frac{T^{15/8}}{X^{7/4}}\cdot \frac{T^{1/4}}{X^{1/4}}=\frac{T^{17/8}}{X^{2}}$. We equate this saving with the diagonal saving of $T$ to optimally choose $X=T^{9/16}$. Hence, in $M_F(T)$, we save $X^{3/2}T^{3/4}=T^{3/2+3/32}$ which gives us $$M_F(T)\ll T^{3/2-3/32+\varepsilon}.$$

	\section{Proof of the Corollaries}\label{S3}
	
	\subsection{Proof of Corollary \ref{cor1}}
	\begin{proof}
		This corollary is an easy consequence of Theorem \ref{maintheorem} if we follow the proof of corollary \cite{Go}. By residue theorem we would express $L(F,1/2+iT_0)$ as 
		\begin{equation*}
			\begin{split}
				L\bigg(F,\frac{1}{2}+iT_0\bigg)^2=&\frac{1}{2\pi i }\int\limits_{(1/\log T_0)}L\bigg(F,\frac{1}{2}+iT_0+s\bigg)^2\frac{e^{s^2}}{s}ds\\&-\frac{1}{2\pi i}\int\limits_{(-1/\log T_0)}L\bigg(F,\frac{1}{2}+iT_0+s\bigg)^2\frac{e^{s^2}}{s}ds.
			\end{split}
		\end{equation*}
		Then, using the functional equation and Stirling's approximation in the second integral, we get 
		\begin{equation}\label{3}
			\begin{split}
				\bigg|L\bigg(F,\frac{1}{2}+iT_0\bigg)\bigg|^2\ll& \int\limits_{-\infty}^{\infty}\bigg|L\bigg(F,\frac{1}{2}+\frac{1}{\log T_0}+i(T_0+t)\bigg)\bigg|^2\times \frac{(1+|t+T_0|^{\frac{3}{\log T_0}})e^{-t^2}}{(t^2+1/ \log^2T_0 )^{1/2}}dt .
			\end{split}
		\end{equation}In a similar fashion for $1/2<\sigma_1<1$ we obtain 
		\begin{equation}\label{4}
			|L(F,\sigma_1+iT_1)|^2\ll 1+\int\limits_{-\infty}^{\infty}\bigg|L\bigg(F,\frac{1}{2}+i(T_1+\tau)\bigg)\bigg|^2\frac{e^{-\tau^2}}{(\tau^2+(\frac{1}{2}-\sigma_1)^2)^{\frac{1}{2}}}d\tau
			.\end{equation}
		Finally in  (\ref{4}) we replace $\sigma_1$ by $1/2+1/\log T_0$ and $T_1$ by $T_0+t$ and apply this bound in (\ref{3}) and estimate the $t$ integral to get
		$$\bigg|L\bigg(F,\frac{1}{2}+iT_0\bigg)\bigg|^2\ll \log T_0+\log T_0\int\limits_{-\infty}^{\infty}\bigg|L\bigg(F,\frac{1}{2}+i(T_0+\tau)\bigg)\bigg|^2e^{-\frac{\tau^2}{2}}d\tau.$$
		Finally truncating the integral at $\tau \in [-\log T_0,\log T_0]$ and estimating the rest of the range trivially we prove the corollary. 
	\end{proof}
	\subsection{Proof of Corollary \ref{cor2}}
	\begin{proof}
		In the proof of \cite{LNQ} we proceed with the assumption  $$\int\limits_{-U}^{U}\big|L\big(F,\frac{1}{2}+it\big)\big|^2 dt\ll_{F,\varepsilon} U^{\frac{3}{2}-\delta+\varepsilon},\ \ \delta>0,$$ instead of the trivial bound of $U^{3/2+\varepsilon}$. Thus, we can replace the bound of (52) of \cite{LNQ} by 
		$$MT^{\varepsilon}\sqrt{R^{+}+U}U^{3/4-\delta/2}\ll \frac{T^{5/4-\delta/2+\varepsilon}}{M^{1/4-\delta/2}},$$
		for the $+$ case. Upper bounds for the diagonal case ($MT^{1+\varepsilon}$) and the - case ($\frac{T^{1+\varepsilon}}{M^{1/2}}+\frac{T^{5/4+\varepsilon}}{M}$) remain unchanged. Thus, we can replace the upper bound of Theorem 1.1 of \cite{LNQ} by 
		$$\sum_{|t_j-T|\leq M}L(F\otimes f_j,1/2)+\int\limits_{T-M}^{T+M}|L(F,1/2+it)|^2\ll MT^{1+\varepsilon}+\frac{T^{5/4-\delta/2+\varepsilon}}{M^{1/4-\delta/2}}.$$
		Then we choose $M=T^{\frac{1-2\delta}{5-2\delta}}$. Thus, we derive the upper bounds 
		\begin{equation*}\label{key}
			L(F,1/2+it)\ll_{F,\varepsilon}(1+|t|)^{\frac{3-2\delta}{5-2\delta}+\varepsilon}, L(F\otimes f_j,1/2)\ll_{F,\varepsilon}(1+|t_j|)^{\frac{6-4\delta}{5-2\delta}+\varepsilon}
			.\end{equation*}
		Finally, we apply Theorem \ref{maintheorem} and put $\delta=\frac{3}{32}$ to conclude the proof. 
	\end{proof}
	\subsection{Proof of Corollary \ref{cor3}}
	\begin{proof}
		We dominate $\sum_{n\leq 2X}\lambda_f(n)^2$ by taking a dyadic partition:
        $$\sum_{n\leq 2X}\lambda_f(n)^2\leq \sum_{j=0}^{\infty}\sum_n \lambda_f(n)^2V(n/2^{-j}X),$$
        with the help of a smooth function $V(x)$ which is supported in $[1-X^{-A},2+X^{-A}]$ and satisfies the decay condition $V^{(j)}\ll X^{jA}$ for $j\geq 1$, $0\leq V(x)\leq 1$ and $V(x)=1$ for $x\in [1,2]$. Similarly, we can bound $\sum_{n\leq 2X}\lambda_f(n)^2$ from below by taking a dyadic partition with the help of a smooth function $W(x)$, supported in $[1,2]$, $0\leq W(x)\leq 1$, $W(x)=1$ in $[1+X^{-A},2-X^{-A}]$ and satisfying the decay condition $W^{(j)}\ll X^{jA}$ for $j\geq 1$. We will only consider the first case and the analysis is similar for the second. 
        
        By inverse Mellin transform we may express the sum $\sum_{n}\lambda_f(n)^2V(n/X)$ (for some $c>2$) as 
		\begin{equation*}\label{key}
			\zeta(2)\sum_{n}\lambda_f(n)^2V(n/X)=\frac{1}{2\pi i}\int_{\Re(s)=c}L(f\otimes f, s)\tilde{V}(s)X^s ds
			,\end{equation*}
		where $\tilde V(s)=\int_0^{\infty}V(x)x^{s-1}dx$. After shifting the line of integral to $\Re(s)=1/2$ we apply the residue theorem and due to the pole at $s=1$ we get the residue $L(\text{Sym}^2f,1)\tilde{V}(1)X$. We note that $\tilde{V}(1)=1+O(X^{-A})$. For the lower bound, we obtain the same main term $L(\text{Sym}^2f,1)\tilde{W}(1)X$ and $\tilde{W}(1)=1+O(X^{-A})$. So we get 
		\begin{equation}\label{3.4}
			\Delta(2X)\ll X^{\varepsilon}\sup_{X_1\ll X}X_1^{1/2}\int_{-\infty}^{\infty}L\bigg(f\otimes f, \frac{1}{2}+it\bigg)\tilde{V}\bigg(\frac{1}{2}+it\bigg)X_1^{it} dt +O(X^{1-A})
			.\end{equation}
		We may truncate the range of $t$. If we integrate $\tilde{V}(1/2+it)$ by parts (Lemma \ref{repeatedint}) repeatedly and use the fact $V^{j}\ll X^{jA}$ then we observe $\tilde{V}(1/2+it)$ is arbitrarily small unless $t\ll X^{A+\varepsilon}$. In that range, we do a dyadic partition of the $t$-integral, thus considering $t\sim T\ll X^{A+\varepsilon}$. For $t\sim T$, by integrating $\int_0^{\infty}V(x)x^{-1/2+it}dx$ by parts, we note that 
		$\tilde{V}(1/2+it)\ll T^{-1}.$
		Putting it in (\ref{3.4}) and by Cauchy-Schwarz inequality we derive
		\begin{equation*}\label{key}
			\begin{split}
    \Delta(2X)=O\bigg[&X^{\varepsilon}\sup_{X_1\ll X}X_1^{1/2}\sup_{T\ll X^{A+\varepsilon}}T^{-1}\bigg(\int_{t\sim T}\bigg|\zeta\bigg(\frac{1}{2}+it\bigg)\bigg|^2dt\bigg)^{1/2}\\
				&\times \bigg(\int_{t\sim T}\bigg|L\bigg(\text{Sym}^2\  f, \frac{1}{2}+it\bigg)\bigg|^2dt\bigg)^{1/2}\bigg]+O(X^{1-A}). 
			\end{split}
			\end{equation*} 
		Assuming $\int\limits_{T}^{2T}\big|L\big(F,\frac{1}{2}+it\big)\big|^2 dt\ll_{F,\varepsilon} T^{\frac{3}{2}-\delta+\varepsilon}$ and using $\int_{t\sim T}|\zeta(1/2+it)|^2dt\ll T^{1+\varepsilon}$ we derive $$\Delta(2X)=O(X^{1/2+A(1/4-\delta/2)})+O(X^{1-A}).$$ We equate these two upper bounds to choose $A=1/(5/2-\delta)$. Thus we derive $$\Delta(2X)\ll X^{1-\frac{1}{5/2-\delta}+\varepsilon}\sim X^{\frac{3-2\delta}{5-2\delta}}.$$ Then by Theorem \ref{maintheorem} we replace $\delta=\frac{3}{32}$ to conclude the proof.  

	\end{proof}
\subsection{Proof of Corollary \ref{cor4}}
\begin{proof}
	The proof is immediate from Theorem \ref{maintheorem}, i.e. after putting $\alpha=\frac{3}{2}-\frac{3}{32}$ in the quoted results. 
\end{proof}

	\section{Preliminaries}
	\subsection{$\mathrm{SL}(3,\mathbb{Z})$ Maass forms }
	We briefly state basic properties of $\mathrm{SL}(3,\mathbb{Z})$ Maass forms here for which one may refer \cite{Gf}. Let $F$ be an $\mathrm{SL}(3,\mathbb{Z})$ Maass form of type $(\nu_1,\nu_2)$ which is also an eigenfunction of all the Hecke operators. Let $A_F(m,n)$ denotes the $(m,n)^{th}$ Fourier coefficients of $F$ normalized such that $A_F(1,1)=1$. Further on, we will drop the subscript $F$ and denote them as $A(m,n)$. We denote the dual form of $F$ by $\tilde{F}$ which is also an Maass form of type $(\nu_2,\nu_1)$ and with $(m,n)^{th}$ Fourier coefficients $A(n,m)=\overline{A(m,n)}$. The Langlands parameters of $F$ are  denoted by $(\alpha_1,\alpha_2,\alpha_3)$ are defined as
	\begin{equation}\label{langlands}
		\alpha_1=-\nu_1-2\nu_2+1,\  \alpha_2=-\nu_1+\nu_2,\  \alpha_3=2\nu_1+\nu_2-1
		.\end{equation}
	and the dual form $\tilde{F}$ has Langlands parameters $(-\alpha_3,-\alpha_2,-\alpha_1)$. $F$ also satisfies a Ramanujan-type bound on average, precisely
	\begin{equation}\label{ramanujanonaverage}
		\sum_{m_1^2m_2\leq x}|A(m_1,m_2)|^2\ll x^{1+\varepsilon}
		.\end{equation} There is an $L$-function associated to $F$, defined by 
	\begin{equation*}\label{key}
		L(F,s)=\sum_{n=1}^{\infty}A(1,n)n^{-s},
	\end{equation*}
	for $\Re(s)>1$, which satisfies the functional equation
	\begin{equation*}\label{functional}
		\begin{split}
        &G(F,s)L(F,s)=G(\tilde{F},1-s)L(\tilde{F},1-s),\\ \text{ where }~~&G(F,s)=\pi^{-3s/2}\Gamma\bigg(\frac{s-\alpha_1}{2}\bigg)\Gamma\bigg(\frac{s-\alpha_2}{2}\bigg)\Gamma\bigg(\frac{s-\alpha_3}{2}\bigg).
        \end{split}
	\end{equation*}
	These $L$-functions also satisfy an approximate functional equation in the critical strip.
	\subsection{Approximate Functional Equation}
	\begin{lemma}[Approximate Functional Equation]\label{approximatefunctionaleqn} For $0<\Re(s)<1$ and any bounded even function $g(u)$, holomorphic in the strip $-4<\Re(u)<4$ and normalized as $g(0)=1$ and $X>0$, 
		$$L(F,s)=\sum_{n=1}^{\infty}\frac{A(1,n)}{n^{s}}V_s\bigg(\frac{n}{X}\bigg)+\frac{G(\tilde{F},s)}{G(F,s)}\sum_{n=1}^{\infty}\frac{A(n,1)}{n^{1-s}}\tilde{V}_{1-s}(nX).$$
        For $\alpha$ satisfying $\Re(s+\alpha_i)>\alpha>0$ for $i=1,2,3$ and for any $A>0$, we have the following bounds on $V_s(y)$ (exact bounds hold for $\tilde{V}_s(y)$ too) and its derivatives w.r.t $y$:$$y^{j}V_s^{(j)}=\delta_j+O\bigg(\frac{y}{\sqrt{\mathfrak{q}_{\infty}}}\bigg)^{\alpha}~~\text{ and }~~y^{j}V_s^{(j)}\ll\bigg(1+\frac{y}{\sqrt{\mathfrak{q}_{\infty}}}\bigg)^{-A}$$
        and if $\Re(s)=\frac{1}{2}$, we also have
		$\frac{G(\tilde{F},s)}{G(F,s)}\ll 1.$
	\end{lemma}
	
    	\subsection{Voronoi summation formula for $\mathrm{SL}(3,\mathbb{Z})$:}
	Following the notations set at the beginning, let $F$ be an $\mathrm{SL}(3,\mathbb{Z})$ Maass form with $(m,n)^{th}$ Fourier coefficient $A(m,n)$ and let $\tilde{F}$ be its dual form with Fourier coefficients $A(n,m)$. Then, we have a summation formula for $A(1,m)$ twisted by additive characters. We precisely follow the expression of Corollary 3.7 of  \cite{GL}, which we summarize in the lemma below.
	\begin{lemma}[Voronoi type summation formula]\label{voronoi}Let $\psi(x)\in C_c^{\infty}(0,\infty)$ and let $a,\bar{a},q\in \mathbb{Z}$ with $(a,q)=1$. 
		Then we have
		\begin{equation*}\label{key}
			\begin{split}
				\sum_{m=1}^{\infty}A(1,m)e\bigg(\frac{m\bar{a}}{q}\bigg)\psi(m)
				=&\frac{q\pi^{-5/2}}{4i}\sum_{\pm}\sum_{m_1|q}\sum_{m_2>0}\frac{A(m_2,m_1)}{m_1m_2} S(a,\pm m_2;qm_1^{-1})\Psi_{0,1}^{\pm}\bigg(\frac{m_2m_1^2}{q^3}\bigg).
			\end{split}
		\end{equation*}	
	\end{lemma}
	In the above lemma   $\Psi_{0,1}^{\pm}\big(\frac{m_2m_1^2}{q^3}\big)=\Psi_0\big(\frac{m_2m_1^2}{q^3}\big)\pm\frac{\pi^{-3}q^3}{m_1^2m_2i}\Psi_1\big(\frac{m_2m_1^2}{q^3}\big)$ consists of four terms. But we will only estimate $\Psi_0(x)$ with the help of the following lemma by \cite{Li1} (Lemma 6.1) and consider only the term $\Psi_0(x)$. The estimate of $\Psi_1(x)$ is quite similar; thus are the remaining three terms.
	\begin{lemma}
		Suppose $\psi(x)$ is a smooth function compactly supported on $[X,2X]$ and  $\Psi_0(x)$ is defined as above, then for any fixed integer $K\geq1$ and $xX\gg1$, we have
		\begin{equation*}\label{key}
			\begin{split}
				\Psi_0(x)=&2\pi^4xi\int\limits_{0}^{\infty}\psi(y)\sum_{j=1}^{K}\frac{c_j\cos(6\pi x^{1/3}y^{1/3})+d_j\sin(6\pi x^{1/3}y^{1/3})}{(\pi^3xy)^{j/3}}dy+O\left((xX)^{\frac{-K+2}{3}}\right),
			\end{split}
		\end{equation*}	
		
		where $c_j$ and $d_j$ are constants depending on $\alpha_i$, in particular, $c_1=0,\ d_1 = \frac{-2}{\sqrt{3\pi}}$.
	\end{lemma}
We note that the oscillatory part $e(3x^{1/3}y^{1/3})$ is independent of the sum over $j$ and the non-oscillatory terms $(\pi^3xy)^{-j/3}$ decrease with increasing $j$ provided $xy\gg1$. So $\sum_{j=1}^{K}(\pi^3xy)^{-j/3}$ is asymptotic to the term $(\pi^3xy)^{-1/3}$, i.e. $j=1$. Hence we take $K$ sufficiently large so that the error term $O\big((xX)^{\frac{-K+2}{3}}\big)$ can be dropped from further consideration. For such $K$, we will consider only the term $j=1$. For $xy\ll 1$, as the term $e(3x^{1/3}y^{1/3})$ is non-oscillatory, it can be absorbed into the smooth function of $\psi(y)$. Thus, we will consider 
\begin{equation}\label{estimateofpsi}
	\Psi_0(y)\asymp \sum_{\pm} 2\pi^4c_{\pm}xi\int\limits_{0}^{\infty} \psi(y)\frac{ e(\pm 3x^{1/3}y^{1/3})}{(\pi^3xy)^{1/3}}\ \ dy
	.\end{equation}
\subsection{Poisson summation formula}
Let $f(x)$ be a compactly supported integrable function and $C(n)$ be a periodic function modulo $q$. Then, by the Poisson summation formula, we get 
\begin{align}
    \sum_{n\in \bz}C(n)f(n)=\frac{1}{q}\sum_{n\in \bz}\sum_{b\bmod q}C(b)e_q(nb)\int_\br f(y)e\left(-\frac{ny}{q}\right)dy.\label{Poissonsum}
\end{align}
In particular, if $C(n)=e_q(an)$, we have 
\begin{align}
    \sum_{n\in \bz}e_q(an)f(n)
    =\sum_{\substack{n\in \bz\\n\equiv -a\bmod q}}\int_\br f(y)e\left(-\frac{ny}{q}\right)dy.\label{Poissonadditive}
\end{align}
	\subsection{Stationary Phase Analysis}
	To treat the oscillatory integrals of one variable, we will use the following lemmas.

     When the phase function does not have a stationary point, we will use the following lemma from \cite{Mu1}. 
    \begin{lemma}\label{repeatedint}
     Let $g(x)$  be a compactly supported smooth function supported in $[a,b]$ satisfying $g^j(x)\ll_{a,b,j} 1$. Let $f(x)$ be a real valued smooth function satisfying $|f'(x)|\geq \Theta_f$ and $|f^{(j)}(x)|\ll \Theta_f$ for $j\geq 2$. Then, for any $j\in \bn$, we have 
     \begin{align}
         \int_{a}^be(f(x))g(x)dx\ll_{a,bj,\varepsilon} \Theta_f^{-j+\varepsilon}. 
     \end{align}
    \end{lemma}
    When the phase function has a unique stationary point, we will use the following lemma from \cite{BKY} by Blomer, Khan, and Young. So, we restate Proposition 8.2 of \cite{BKY} below.  
	\begin{lemma}\label{StationaryPhase1}
		Let $0<\delta<1/10, \Theta_g,\Theta_f,\Omega_g,L,\Omega_f>0$ and let $Z:=\Omega_f+\Theta_f+\Theta_g+L+1$ and we also assume that 	
		\begin{equation}\label{8.7}
			\Theta_f\geq Z^{3\delta},\ \ L\geq \Omega_g\geq \frac{\Omega_fZ^{\delta/2}}{\Theta_f^{1/2}}
			.\end{equation}
		Let $g(x)$ be a compactly supported smooth function with support in a length $L$ and satisfying the derivative $g^{(j)}(x)\ll \Theta_g\Omega_g^{-j}$ and let $x_0$ be the unique point such that  $f'(x_0)=0$, where $f(x)$ is a smooth function satisfying 
		\begin{equation}\label{8.8}
			f''(x)\gg \Theta_f\Omega_f^{-2},\ \ f^{(j)}(x)\ll \Theta_f \Omega_f^{-j},\ \ \ \forall j\in \mathbb{N}
			.\end{equation} 
		Then the oscillatory integral $I=\int\limits_{-\infty}^{\infty}g(x)e(f(x))dx$ would have the asymptotic expression (for arbitrary $A>0$)
		\begin{equation}\label{8.9}
			I=\frac{e(f(x_0))}{\sqrt{f''(x_0)}}\sum_{n\leq 3\delta^{-1}A}p_n(x_0)+O_{A,\delta}(Z^{-A}),
		\end{equation}
		where \begin{equation}\label{8.10}
			\begin{split}p_n(x_0)&=\frac{e^{\pi i/4}}{n!}\bigg(\frac{i}{2f''(x_0)}\bigg)^n G^{(2n)}(x_0)\\ \text{ where } G(x)&=g(x)e(f(x)-f(x_0)-f''(x_0)(x-x_0)^2/2).\end{split}
		\end{equation}
		Each $p_n$ is a rational function in derivatives of $f$ satisfying 
		\begin{equation}\label{8.11}
			\frac{d^j}{dx_0^j}p_n(x_0)\ll \Theta_g(\Omega_g^{-j}+\Omega_f^{-j})((\Omega_g^2\Theta_f/\Omega_f^2)^{-n}+\Theta_f^{-n/3})
			.\end{equation}
	\end{lemma}
	\begin{remark}
		As observed in \cite{BKY}, from (\ref{8.7}) and (\ref{8.11}), in the asymptotic expression (\ref{8.10}), every term is smaller than the preceding term. So it is enough to consider the leading term in the asymptotic provided we verify \ref{8.7}   
	\end{remark}
    For stationary phase analysis on  multivariate integrals,  we will mention Theorem 7.7.1 and Lemma 7.7.5 of 
\cite{Ho}, . Here, we will only mention the $\mathbb{R}^2$ version of the lemma stated in \cite{HMQ}. 
\begin{lemma}[Theorem 7.7.1]\label{Hormandertheorem}
Let $K\subset\mathbb{R}^2$ be a compact set and let $X$ be an open subset of $\mathbb{R}^2$ containing $K$ and let $k$ be a non-negative integer. Let $u\in C_c^k(K)$ and $f\in C^{k+1}(X)$ where $f$ is a real valued bounded function. Then, for $\lambda>0$, 
\begin{equation*}\label{key}
\bigg|\int_{K}e(\lambda f(x))u(x)dx\bigg|\leq C \lambda^{-k}\sum_{j_1+j_2\leq k}\sup|\partial_1^{j_1}\partial_2^{j_2}u||f'|^{j_1+j_2-2k}
.\end{equation*}
where $C$ is bounded as long as $f$ stays bounded in $C^{k+1}(X)$. 
\end{lemma}
\begin{lemma}[Lemma 7.7.5]\label{Hormanderlemma }
Let $K\subset\mathbb{R}^2$ be a compact set and let $X$ be an open subset of $\mathbb{R}^2$ containing $K$. Let $u\in C_c^2(K)$ and $f\in C^4(X)$ where $f$ is a real valued function. Then if there is a point $x_0\in K$ such that $f'(x_0)=0$ and $\det(H_f(x_0))\neq 0$ and $f'(x)\neq 0\ \ ; \forall x\in K\setminus x_0$ then for $\lambda>0$ we have 
\begin{equation}\label{Hormanderlemmaequation}
\bigg|\int_{K}e(\lambda f(x))u(x)dx-\frac{u(x_0)e(\lambda f(x_0))}{\lambda\sqrt{-\det(H_f(x_0))}}\bigg|\leq \frac{C}{\lambda^2}\left(1+|\det f''(x_0)|^{-3}\right)\sum_{j_1+j_2<4}\sup|\partial_1^{j_1}\partial_2^{j_2}u|		
.\end{equation}
\end{lemma}
	\section{Setting up for delta method }\label{S5}
	\subsection{Adjusting the range of $n$:}
	We will consider the limit of integration for the second moment (\ref{FirstDefnofSecondMoment})  to be $[2T,3T]$ instead of $[T,2T]$, which would not affect the end result. We start with the application of the approximate functional equation on $M_F(T)$
	\begin{equation*}
		\begin{split}
			M_F(T)=\int\limits_{2T}^{3T}\bigg|\sum_{n=1}^{\infty}\frac{A(1,n)}{n^{1/2+it}}V_{1/2+it}(n)+\frac{G(\tilde{F},s)}{G(F,s)}\sum_{n=1}^{\infty}\frac{A(n,1)}{n^{1/2-it}}\tilde{V}_{1/2-it}(n)\bigg|^2dt
		\end{split}	
		.\end{equation*} 
	As $\frac{G(\tilde{F},s)}{G(F,s)}\ll 1$ for $\Re(s)=1/2$ and both the sums are essentially the same for $\Re(s)=1/2$, we can just consider the first sum.
	And as we are only considering $t$-aspect here, by Lemma \ref{approximatefunctionaleqn} we may restrict the range to $n\ll t^{3/2+\varepsilon}$ getting
	\begin{equation*}\label{key}
		M_F(T)
		\ll\int\limits_{2T}^{3T}\bigg|\sum_{n\ll t^{3/2+\varepsilon}}\frac{A(1,n)}{n^{1/2+it}}\bigg|^2dt
		.\end{equation*}
	Then, we take a dyadic partition for the sum over $n$ with a smooth function $V_1(x)$ supported in $[1,2]$ and satisfying $V_1^j(x)\ll1$, and arrive at the expression
	\begin{equation}\label{2.1}
		M(T)=\int\limits_{2T}^{3T}\bigg|\sum_{n\in \mathbb{N}}A(1,n)n^{-it}V_1\bigg(\frac{n}{N}\bigg)\bigg|^2dt, \text{ and }M_F(T)\ll \sup_{N\ll T^{3/2+\varepsilon}}N^{-1+\varepsilon}M(T).
	\end{equation} 
    For small values of $N$ ($N\ll T^{15/11+\varepsilon}$), we will use the trivial bound of $M(T)\ll N^{2+\varepsilon}$. Hence, 
    \begin{equation}
		M_F(T)\ll \sup_{T^{15/11+\varepsilon}\ll N\ll T^{3/2+\varepsilon}}N^{-1+\varepsilon}M(T)+T^{15/11+\varepsilon}.
	\end{equation} 
	\subsection{Breaking down the integral}
	We fix a variable $0<X\ll T^{1-\varepsilon}$ here, whose optimal value would be chosen later. We choose a non-negative smooth function $V(x)$ supported in $[-1/2,3/2]$ and $1$ in the range $[0,1]$ to split the integral of (\ref{2.1}) into small integrals of the size $\xi$ where $X\leq \xi\leq 2X$.
   Then, we take an average over $\xi$ with the help of a non-negative smooth function $V_2$ supported in $[1,2]$ satisfying $V_2^j(x)\ll 1$ and $\int_xV_2(x)=1$.  and where we vary $\xi$ continuously in the range $[X,2X]$. Let $\mathbbm 1_S$ be the indicator function of $S$. Then, 
    $$\mathbbm 1_{2T\leq t\leq 3T}\leq \frac{1}{X}\sum_{\lfloor\frac{T}{2X}\rfloor\leq r\leq\lfloor\frac{2T}{X}\rfloor} \int\limits_\br V\left(\frac{t-T-r\xi}{\xi}\right) V_2\left(\frac{\xi}{X}\right),$$
    for $t\in\br$. We summarize the splitting of the integral in the following lemma.
    \begin{lemma}
    Let $M(T)$ be as in (\ref{2.1}) and let $V(x)$ be a non-negative smooth function supported in $[-1/2,3/2]$ and $1$ in the range $[0,1]$, satisfying $V^j(x)\ll 1$ and $V_2(x)$ be a non-negative smooth function supported in $[1,2]$ satisfying $V_2^j(x)\ll 1$ and $\int_\br V_2(x)dx=1$. Then,  
    \begin{equation*}\label{2.3}
		M(T)\ll X^{-1}\sum_{T/2X\leq r\leq 2T/X}	\mathcal{M}(r,T),
	\end{equation*}
	where 	\begin{equation}\label{2.4}
		\mathcal{M}(r,T):=\int\limits_{\mathbb{R}}\int\limits_{\mathbb{R}}\bigg|\sum_{n}A(1,n)n^{-it}V_1\bigg(\frac{n}{N}\bigg)\bigg|^2V\bigg(\frac{t-T-r\xi}{\xi}\bigg)V_2\bigg(\frac{\xi}{X}\bigg)dtd\xi.
	\end{equation}
    \end{lemma}
	\begin{remark}$M_F(T)\ll N^{-1}M(T)\ll (NX)^{-1}\sum_{r\sim T/X}\mathcal{M}(r,T)$\end{remark}
	\subsection{$t$- integral}
	\begin{lemma}\label{tintegrallemma} Let $H:=\frac{N}{X}$and $X\gg T^{1/2+\varepsilon}$. For some smooth function $U(x,y)$ of two variables supported in $[-H T^{\varepsilon}, HT^{\varepsilon}]\times [N,2N]$ and satisfying 
    $$x^iy^j\frac{\partial^i\partial ^j }{\partial x^i\partial y^j}U(x,y)\ll_{i,j,\varepsilon} \left(\frac{x}{H}\right)^iT^{(i+j)\varepsilon},$$
    we have 
    \begin{align}
        \mathcal{M}(r,T)\asymp X\cdot \mathfrak{M}(r)+O(T^{-A}),\label{2.11}
    \end{align}
    where 
    \begin{equation}\label{2.13}
\mathfrak{M}(r):=\int\limits_{\mathbb{R}}\sum_{h}\sum_{n}A(1,n)\overline{A(1,n+h)}e\bigg(\frac{(T+r\xi)}{2\pi}\times \frac{h}{n}\bigg)U(h,n)V_1\left(\frac{n}{N}\right)V_2\bigg(\frac{\xi}{X}\bigg)d\xi
		.\end{equation}
	\end{lemma}
	\begin{proof}
		We open the absolute value square in (\ref{2.4}) and denote the copy of the $n$ variable by $m$. Further, we may write $m=n+h$. Then, we perform a change of variable $u:=(t-T-r\xi)/\xi$ on the $t$-integral and get  
		\begin{equation*}\label{key}
			\begin{split}
				\mathcal{M}(r,T)=&\int\limits_{\mathbb{R}}\xi\sum_{h}\sum_{n}A(1,n)\overline{A(1,n+h)}\bigg(1+\frac{h}{n}\bigg)^{i(T+r\xi)}\\&\times \left[\int_{\mathbb{R}}\bigg(1+\frac{h}{n}\bigg)^{iu\xi}V(u)du\right]V_1\bigg(\frac{n}{N}\bigg)V_1\bigg(\frac{n+h}{N}\bigg)V_2\bigg(\frac{\xi}{X}\bigg)d\xi.
			\end{split}
		\end{equation*}
        By repeated integration by parts (Lemma \ref{repeatedint}) , we observe that the $u$ integral is negligibly small unless 
        $$|X\log (1+h/n)|\ll T^{\varepsilon}\implies |h|\ll \frac{NT^{\varepsilon}}{X}.$$
We denote $H=\frac{N}{X}$. Now, 
Let $V_3(u)$ be an even smooth function supported in $[-3/2,3/2]$ with value $1$ in $-[1,1]$ satisfying $V_3^{j}(y)\ll 1$. Then, if we define $$U(h,n):=\left[\int_{\mathbb{R}}\big(1+\frac{h}{n}\big)^{iu\xi}V(u)du\right] V_3\left(\frac{h}{HT^{\varepsilon}}\right)V_1\bigg(\frac{n+h}{N}\bigg)$$ 
        it satisfies 
        $$x^iy^j\frac{\partial^i\partial ^j }{\partial x^i\partial y^j}U(x,y)\ll_{i,j,\varepsilon} \left(\frac{x}{H}\right)^iT^{(i+j)\varepsilon}.$$
$V_3(h/HT^{\varepsilon})$ restricts the support of $x$ of $U(x,y)$ in $[-HT^{\varepsilon},HT^{\varepsilon}]$ with negligible error term.  

The exponential term $(1+h/n)^{i(T+r\xi)}$ has the phase function of the form $(2\pi)^{-1}(T+r\xi)\log(1+h/n)$. Upon expressing $\log(1+h/n)$ in terms of the Taylor series, we note that the error term $$(T+r\xi)[\log(1+h/n)-h/n]\ll \frac{hT^{\varepsilon}}{H}$$ provided 
		$$T\frac{h^2}{n^2}\ll \frac{hT^{\varepsilon}}{H}\impliedby  T\cdot \frac{NT^{\varepsilon}}{NX}\ll X\iff X\gg T^{1/2+\varepsilon}.$$
		Thus, assuming $X\gg T^{1/2+\varepsilon}$ (though at a later stage, this condition would be proven to be a necessary one), we can drop the higher order terms of the Taylor series by absorbing them into the smooth function $U(h,n)$ and only use the term $e\left(\frac{(T+r\xi)}{2\pi}\times \frac{h}{n}\right).$ Finally, absorbing $\frac{\xi}{X}$ into the smooth function $V_2(\xi/X)$, we conclude the proof of the lemma. 
	\end{proof}    
	In the next section, we will apply the $\delta$ method of Duke, Friedlander, and Iwaniec to separate the oscillations in $\mathfrak{M}$ (\ref{2.13}). For brevity, we will write $\mathfrak{M}$ instead of $\mathfrak{M}(r)$.
	
	\section{Application of the delta method}\label{S6}
	In order to separate the oscillations in the expression of $\mathfrak{M}$ in (\ref{2.13}), we will employ the circle method, specifically the $\delta$-method of Duke, Friedlander, and Iwaniec (Chapter $20$ of \cite{IK}). Here $\delta$ symbol represents the function
	\begin{equation*}\label{key}
		\delta(n):\mathbb{Z}\rightarrow\{0,1\},\ \  \text{such that }\ \ \delta(n)=\begin{cases}
			0 \text{ if } n\neq 0\\
			1\text{ if } n=0
		\end{cases}
		.\end{equation*}
	We will use the expression of $\delta(n)$ mentioned in \cite{Mu2}. We summarize the expression and a few of its properties in the following lemma.
	\begin{lemma}\label{delta}
		\begin{equation*}
			\delta(n)=\frac{1}{Q}\sum_{1\leq q\leq Q}\frac{1}{q}\ \sideset{}{^*}\sum_{a\bmod q}e\bigg(\frac{an}{q}\bigg)\int\limits_{-\infty}^{\infty}g(q,x)e\bigg(\frac{nx}{qQ}\bigg)dx,
		\end{equation*} 
		where the sum over $a$ is over the reduced residue class of $q$ (signified by $*$)  and for any $\alpha>1$, $g(q,x)$ satisfies
		\begin{align}
				&	g(q,x)=1+O\bigg(\frac{Q}{q}\bigg(\frac{q}{Q}+|x|\bigg)^{\alpha}\bigg),\ \ \ g(q,x)\ll|x|^{-\alpha},\label{gqxshape}\\
                & x^j\frac{\partial ^j}{\partial x^j}g(q,x)\ll \log Q\min\left\{\frac{Q}{q},\frac{1}{|x|}\right\},\label{gqxderivative}\\
                & \int_{\br} (|g(q,x)|+|g(q,x)|^2)dx\ll Q^{\varepsilon}.\label{gqxintegral}
			\end{align}
	\end{lemma}
	In $\mathfrak{M}$ (\ref{2.13}) we introduce an additional sum over $m$ accompanied with $\delta(m-n-h)W\big(\frac{m-n-h}{Q_0^2}\big)$. Here $\delta(m-n-h)$ represents the condition $m=n+h$ and $W$ is an even smooth function supported in $[-1,1]$ with $W(0)=1$
	. Then, we apply (\ref{delta}) with modulus $Q_0=\sqrt{N/(T/X)}$. We introduce a dyadic partition of unity in the $q$ sum and separate the $n$ sum and $m$ sum to get 
	\begin{equation}\label{2.17}
		\begin{split}
			\mathfrak{M}\ll&\frac{T^{\varepsilon}}{Q_0}~\sup_{Q\ll Q_0}~\sum_{q\sim Q}~\frac{1}{q}~\ \sideset{}{^*}\sum_{a\bmod q}~\int_{\mathbb{R}}V\bigg(\frac{\xi}{X}\bigg)~\sum_{h}e_q(-ah)\\
			&\times\bigg[\sum_{n}\bar A(1,n)e_q(-an)e\bigg(\frac{(T+r\xi)}{2\pi}\times \frac{h}{n}\bigg)U(h,n)V\left(\frac{n}{N}\right)\bigg]\\
		&\times\bigg[\sum_{m}A(1,m)e_q(am)\int_{\mathbb{R}}g(q,x)e\bigg(\frac{(m-n-h)x}{qQ_0}\bigg)W\bigg(\frac{m-n-h}{Q_0^2}\bigg)dx\bigg]\ d\xi.
		\end{split}
	\end{equation}
	\begin{remark}
	    From now on, $V(x)$ is any smooth function supported in $[1,2]$ satisfying $V^j(x)\ll 1$. 
	\end{remark}
	
	\section{Application of the Voronoi type Summation Formula}\label{S7}
\subsection{Voronoi summation in the $m$-sum}
We will apply the Voronoi summation formula to the $m$-sum
\begin{equation*}	\begin{split}&S_\cM:=\sum_{m}A(1,m)e_q(am)\int_{\mathbb{R}}g(q,x)e\bigg(\frac{(m-n-h)x}{qQ_0}\bigg)W\bigg(\frac{m-n-h}{Q_0^2}\bigg)dx.
	\end{split}
\end{equation*}
Using the expression of (\ref{estimateofpsi}), we evaluate 
\begin{align}
    &\Psi_0\left(\frac{m_1^2m_2}{q^3}\right)\asymp \
    \left(\frac{m_1^2m_2}{q^3}\right)^{2/3}\sum_{\pm}c_{\pm}I_{1,\pm} \label{Psi07.5},\\ \text{where }&I_{1,\pm
    }=\int_{-\infty}^{\infty}g(q,x)\int_0^\infty  u^{-1/3} e\left(\frac{x(u-n-h)}{qQ_0}\pm\frac{3(m_2m_1^2u)^{1/3}}{q}\right)\nonumber\\&\hspace{7cm}\times W\left(\frac{u-n-h}{Q_0^2}\right)dudx.
\end{align}
We derive 
\begin{lemma}\label{L7.3}
When $m_2m_1^2\ll \frac{N^{2+\varepsilon}}{Q_0^3}$,
\begin{equation*}\label{key}
\Psi_0\bigg(\frac{m_2m_1^2}{q^3}\bigg)\asymp\bigg(\frac{m_2m_1^2}{q^3}\bigg)^{2/3}\sum_{\pm}c_{\pm} \frac{Q_0q}{N^{1/3}}e\bigg(\pm \frac{3(m_2m_1^2)^{1/3}(n^{1/3}+\frac{h}{3n^{2/3}})^{1/3}}{q}\bigg)\int_x g(q,x),
\end{equation*}
with negligible error term,  and $\Psi_0\left(\frac{m_2m_1^2}{q^3}\right)$ is negligibly small otherwise. 
\end{lemma}
\begin{proof}
We start with a change of variable  $u=qQ_0y+n+h$ in the integral $I_{\pm}$  in (\ref{Psi07.5})
\begin{align}
    I_{1,\pm}= &qQ_0\int_{-\infty}^\infty  (n+h+qQ_0y)^{-1/3}\nonumber e\left(\pm\frac{3(m_2m_1^2)^{1/3}(n+h+qQ_0y_1)^{1/3}}{q}\right)\\&\times \int_{\br}g(q,x)e(xy)W\left(\frac{y}{Q_0/q}\right)dx dy.\nonumber
\end{align}
When $q\gg Q_0^{1-\varepsilon}$, by $W\left(\frac{y}{Q_0/q}\right)$, we get $y\ll T^{\varepsilon}$. When $q\ll Q_0^{1-\varepsilon}$, by (\ref{gqxshape}), $g(q,x)$ can be approximated by $1$ with $x$ supported in $[-T^{\varepsilon}, T^{\varepsilon}]$, with a negligible error term. Hence, the $x$ integral would be negligibly small unless $y\ll {T^{\varepsilon}}$. Thus, for any value of $q$, we get $y\ll T^{\varepsilon}$ and in that range $e(xy)$ can be absorbed into the smooth functions.

At this point, by repeated integration by parts (Lemma \ref{repeatedint}), we observe that the $y$ integral is negligibly small unless $m_1^2m_2\ll \frac{N^{2+\varepsilon}}{Q_0^3}$. Then, by the Taylor series expansion of $\left(1+\frac{qQ_0y}{n+h}\right)^{1/3}$ we expand 
the above exponential term:
\begin{align*}
    \frac{3(m_2m_1^2)^{1/3}(n+h)^{1/3}}{q}+\frac{(m_2m_1^2)^{1/3}Q_0qy}{q(n+h)^{2/3}}+\cdots 
\end{align*}
When $m_1^2m_2\ll \frac{N^{2+\varepsilon}}{Q_0^3}$, all but the first term along with the $y$ integral  can be absorbed into the smooth function $U(h,n)$ and we arrive at 
$$I_{1,\pm}\asymp \frac{qQ_0}{N^{1/3}}
 e\left(\pm \frac{3(m_2m_1^2)^{1/3}(n+h)^{1/3}}{q}\right)\int_{\br}g(q,x) V(\dots )dx+O(T^{-A}).$$
Now, we also expand $(1+h/n)^{1/3}$ by its Taylor series expansion
$$\left(1+\frac{h}{n}\right)^{1/3}=1+\frac{h}{3n}-\frac{h}{9n^2}+\cdots.$$
As
$$X\gg \sqrt{T}\text{ and } m_1^2m_2\ll \frac{N^{2+\varepsilon}}{Q_0^3}\implies \frac{(m_1^2m_2)^{1/3}n^{1/3}h^2}{n^{2}q}\ll \frac{hT^{\varepsilon}}{H},$$
all but the first two terms corresponding to the above Taylor series expansion can be absorbed into the smooth function $U(h,n)$ and we arrive at our final expression.
\end{proof}

Thus, the sum over $m$ transforms into 
\begin{align}
        S_\cM\asymp \frac{Q_0}{N^{1/3}}\sum_{\pm}\underset{m_1^2m_2\ll \frac{N^{2+\varepsilon}}{Q_0^3}}{\sum_{m_1|q}\sum_{m_2>0}}&\frac{A(m_2,m_1)\cdot m_1^{1/3}}{m_2^{1/3}}S(\bar a, \pm m_2, qm_1^{-1})\nonumber\\&\times e\left(\pm \frac{3(m_2m_1^2)^{1/3}(n^{1/3}+\frac{h}{3n^{2/3}})}{q}\right)\int g(q,x)dx +O(T^{-A}).\label{msumaftervoronoi}
    \end{align}
\subsection{Voronoi summation formula on the $n$-sum}
After computing the $m$ sum, the $n$ sum for a fixed $m=m_1^2m_2$ becomes
\begin{equation}\label{2.23}
\begin{split}
&S_\cN:=\sum_{n}A(1,n)e_q(an)e\bigg(\frac{(T+r\xi)}{2\pi}\times \frac{h}{n}\pm \frac{3(m_1^2m_2)^{1/3}(n^{1/3}+\frac{h}{3n^{2/3}})}{q}\bigg)U(h,n)V\left(\frac{n}{N}\right).
\end{split}
\end{equation}
Similar to the $m$ sum, we will apply the Voronoi summation formula to the $n$-sum and only consider the part corresponding to $\Psi_0$. For that, we use (\ref{estimateofpsi}) with $x=n_2n_1^2/q^3$. Thus, 
$$\Psi_0\bigg(\frac{n_1^2n_2}{q^3}\bigg)\asymp\bigg(\frac{n_1^2n_2}{q^3}\bigg)^{2/3}\cdot I_2,$$
where the oscillatory integral $I_2$ is 
\begin{equation}\label{Iintegralnsum}
\begin{split}
I_2:=\int\limits_{0}^{\infty}e\bigg(&\frac{(T+r\xi)}{2\pi}\times \frac{h}{u}\pm \frac{(m_1^2m_2)^{1/3}h}{qu^{2/3}}\\&+\frac{3(\pm (m_1^2m_2)^{1/3}\pm(n_1^2n_2)^{1/3})u^{1/3}}{q}\bigg)\frac{1}{u^{1/3}}U(h,u)V(u/N)du.
\end{split}
\end{equation}
Note that all the choices of $\pm$s are allowed, but we fix one to begin with. As $(m_1^2m_2)\ll \frac{N^{2+\varepsilon}}{Q_0^3}$, by repeated integration by parts (Lemma \ref{repeatedint}), this integral is negligibly small unless $(n_1^2n_2)\ll \frac{N^{2+\varepsilon}}{Q_0^3}$. With $I_2$ as defined in (\ref{Iintegralnsum}), the $n$ sum (\ref{2.23}) transforms into 
\begin{equation}\label{nsumaftervoronoi}
\begin{split}
S_\cN\asymp \frac{1}{q}\sum_{\pm}\underset{n_1^2n_2\ll \frac{N^{2+\varepsilon}}{Q_0^3}}{\sum_{n_1|q}n_1^{1/3}\sum_{n_2>0}}\frac{A(n_1,n_2)}{n_2^{1/3}}S(\bar a,\pm n_2;qn_1^{-1})\cdot I_2+O(T^{-A})
\end{split}
\end{equation}

\subsection{Poisson summation formula on the $h$ sum}
Now, we evaluate the following $h$ sum:
\begin{align}
S_\cH=&\sum_{h}e_q(-ah) I_2(h)\label{Shinitial}
\end{align}
where, $I_2$ is the integral (\ref{Iintegralnsum}) and we have written $I_2(h)$ to show the dependency on $h$. Applying the Poisson Summation formula on $S_\cH$, we get  
\begin{lemma}\label{Shfinallemma}
\begin{align}
    S_\cH\asymp \frac{N^{5/3}}{T} \sum_{\substack{h\sim \frac{qT}{N}\\h\equiv a\bmod q}} e\bigg(\frac{3(T+r\xi)^{1/3}( \pm(n_1^2n_2)^{1/3}\pm(m_1^2m_2)^{1/3} )}{(2\pi)^{1/3}q^{2/3}h^{1/3}}\bigg) V(\cdots )+O(T^{-A}).\label{Shfinal}
\end{align}
\end{lemma}
\begin{proof}
Application of the Poisson summation formula (\ref{Poissonadditive}) (and denoting the dual variable by $h$ too ) on $S_\cH$ (\ref{Shinitial}) gives us 
\begin{equation*}\label{key}
\begin{split}
S_\cH&=\sum_{h\equiv a\bmod q}\int_\br I_2(y)e(-hy/q)dy= \sum_{h\equiv a\bmod q} I_3,
\end{split}
\end{equation*}
where the double integral $I_3$ is given by 
\begin{align}
    I_3:=\int\limits_{0}^{\infty}e\bigg(\frac{3(\pm (m_1^2m_2)^{1/3}\pm(n_1^2n_2)^{1/3})u^{1/3}}{q}\bigg)\frac{1}{u^{1/3}}V(u/N)\cdot I_{h,2}\ du,\label{I3hsum}
\end{align}
where, after a change of variable $y\mapsto \frac{yu}{X}$, we have   
\begin{align}
    I_{h,2}=\frac{u}{X}\int_\br e\bigg(\frac{(T+r\xi)}{2\pi}\times \frac{y}{X}\pm \frac{(m_1^2m_2)^{1/3}yu^{1/3}}{qX}\bigg) e\left(-\frac{hyu}{qX}\right)U\left(\frac{yu}{X},u\right)dy.
\end{align}
As $(a,q)=1$ and $h\equiv a\bmod q$, the contribution of $S_h$ would be negligibly small unless $|h|\geq 1$. As $u\sim N$, $(m_1^2m_2)^{1/3}\ll \frac{N^{2/3}T^{\varepsilon}}{Q_0}$, and $\frac{\partial ^j}{\partial y^j}U\left(\frac{yu}{X},u\right)\ll \left(\frac{uT^{\varepsilon}}{XH}\right)^j\ll T^{j\varepsilon},$  by repeated integration by parts (Lemma \ref{repeatedint}), the $y$ integral is negligibly small unless $h\sim \frac{qT}{N}$. With $h$ free in this range, we get that the integral is negligibly small unless
\begin{align}
 &\left|\frac{(T+r\xi)q}{2\pi h}\pm \frac{(m_1^2m_2)^{1/3}u^{1/3}}{h}-u\right|\ll T^{2\varepsilon}\frac{qX}{h}\ll \frac{NT^{2\varepsilon}}{T/X}.\nonumber
\end{align}
As $|h|\geq 1$,  $\frac{(m_1^2m_2)^{1/3}u^{1/3}}{h}\ll \frac{N}{Q_0}\ll \frac{N}{T/X}$ ( provided  $Q_0>T/X$, which we will verify with our final choice of $Q_0$). So, we can drop this term from further consideration. Now, we change the variable $u$ to $u_1$, where  $u_1:=u-\frac{(T+r\xi)q}{2\pi h}$ and we have $|u_1|\ll\frac{NT^{2\varepsilon}}{T/X}$. 
In that range of $u_1$, $I_{h,2}$ can be assumed to be a smooth function with absolute value bounded by $\frac{N}{X}$. Consequently, we can write $I_3$ as 
\begin{align}
    I_3\asymp \frac{N^{2/3}}{X}\int_{u_1\ll \frac{NT^{\varepsilon}X}{T}} e\bigg(\frac{3(\pm (m_1^2m_2)^{1/3}\pm(n_1^2n_2)^{1/3})(u_1+\frac{(T+r\xi)q}{2\pi h})^{1/3}}{q}\bigg)du_1+O(T^{-A}).\nonumber
\end{align}
By repeated integration by parts (Lemma \ref{repeatedint}), the integral $I_3$ is negligibly small unless 
$$|\pm (m_1^2m_2)^{1/3}\pm(n_1^2n_2)^{1/3}|\ll \frac{qT^{\varepsilon}T}{N^{1/3}X}.$$
Then, in the Taylor series expansion of $\left(1+\frac{u_1}{\frac{(T+r\xi)q}{2\pi h}}\right)^{1/3}$, the second and higher-order terms can be incorporated into the smooth functions as
$$\frac{3(T+r\xi)^{1/3}( \pm(m_1^2m_2)^{1/3} \pm(n_1^2n_2)^{1/3})}{(2\pi)^{1/3}q^{2/3}h^{1/3}}\cdot \frac{u_1}{\frac{3(T+r\xi)q}{2\pi h}}\ll \frac{N^{1/3}}{q}\cdot  \frac{qT^{\varepsilon}T}{N^{1/3}X}\cdot \frac{X}{T}\ll T^{\varepsilon} .$$
So, we only consider the first term of the Taylor series expansion and finally arrive at 
\begin{align}
    I_3\asymp \frac{N^{2/3}T^{\varepsilon}}{X}\cdot \frac{NX}{T}\cdot  e\left(\frac{3(T+r\xi)^{1/3}( \pm(m_1^2m_2)^{1/3} \pm(n_1^2n_2)^{1/3})}{(2\pi)^{1/3}q^{2/3}h^{1/3}}\right) V(\cdots )+O(T^{-A}).
\end{align}
Putting it back into $S_\cH$ (\ref{Shinitial}), we conclude the proof of the lemma.
\end{proof}

\subsection{Final Equation After Voronoi}
We put (\ref{msumaftervoronoi}), (\ref{nsumaftervoronoi}), (\ref{Shfinal}) into (\ref{2.17}) and arrive at the expression 
\begin{equation}
\begin{split}
\mathfrak{M}\ll& \frac{T^{\varepsilon}}{Q_0}\cdot \frac{Q_0}{N^{1/3}}\cdot \frac{N^{5/3}}{T}\sup_{Q\ll Q_0}\frac{1}{Q^2}\sum_{q\sim Q} \int_{\xi\sim X}\sum_{h\sim\frac{qT}{N}} \int g(q,x)dx\\
&\underset{m_1^2m_2\ll \frac{N^{2+\varepsilon}}{Q_0^3}}{\sum_{m_1|q}\sum_{m_2>0}}\frac{A(m_2,m_1)\cdot m_1^{1/3}}{m_2^{1/3}}S(\bar a,\pm  m_2, qm_1^{-1})e\bigg(\frac{3(T+r\xi)^{1/3}( \pm(m_1^2m_2)^{1/3} )}{(2\pi)^{1/3}q^{2/3}h^{1/3}}\bigg)\\
&\underset{n_1^2n_2\ll \frac{N^{2+\varepsilon}}{Q_0^3}}{\sum_{n_1|q}\sum_{n_2>0}}\frac{A(n_1,n_2)n_1^{1/3}}{n_2^{1/3}}S(\bar a,\pm n_2;qn_1^{-1})e\bigg(\frac{3(T+r\xi)^{1/3}( \pm(n_1^2n_2)^{1/3})}{(2\pi)^{1/3}q^{2/3}h^{1/3}}\bigg)
\end{split}
.\end{equation}
Now, we evaluate the $x$ integral by (\ref{gqxintegral}), we take a dyadic partition of both the $m_1^2m_2$ sum and the $n_1^2n_2$ sum and then by Cauchy's inequality and the symmetry of the $n_1,n_2$ and $m_1,m_2$ sum, we get 
\begin{align}
    \mathfrak{M}\ll& \frac{N^{4/3}T^{\varepsilon}}{T}\sup_{Q\ll Q_0}\frac{1}{Q^2}\sup_{M_0\ll \frac{N^{2+\varepsilon}}{Q_0^3}}M_0^{-2/3}\sum_{q\sim Q} \int_{\xi\sim X}\sum_{h\sim\frac{qT}{N}} \nonumber \\
    &\times \Bigg|\underset{m_1^2m_2\sim M_0}{\sum_{m_1|q}\sum_{m_2>0}}A(m_2,m_1)\cdot m_1\cdot S(\bar h, m_2, qm_1^{-1})e\bigg(\frac{3(T+r\xi)^{1/3} (m_1^2m_2)^{1/3} }{(2\pi)^{1/3}q^{2/3}h^{1/3}}\bigg)\Bigg|^2\label{lastequationbeforelargesieve}
\end{align}
Here, we have only considered the $+$ sign case as the analysis of the other cases is exactly similar.  
\section{Duality principle of the large sieve}\label{S8}
To interchange the order of summations in the above expression, we will use the duality principle of the large sieve. We restate the duality principle stated in Chapter 7.1 of \cite{IK}. 
\begin{lemma}\label{largesieveduality}
For $1\leq m\leq M$ and $1\leq n\leq N$ let $\alpha_m,\beta_n\in \mathbb{C}$ and $(\phi(m,n))_{\underset{1\leq n\leq N}{1\leq m\leq M}}$  be a complex matrix . Then, if we have 
$$\sum_{n}\big|\sum_{m}\alpha_m\phi(m,n)\big|^2\ll\Delta||\alpha||^2,$$
for the same $\Delta$ we would have 
$$\sum_{m}\big|\sum_{n}\beta_n\phi(m,n)\big|^2\ll\Delta||\beta||^2.$$	
\end{lemma}
Using the duality principle and the Ramanujan type bound of average (\ref{ramanujanonaverage}), we get 
\begin{lemma}\label{Largesievemain}
$$\mathfrak{M} \ll\frac{N^{4/3}T^{\varepsilon}}{T}\sup_{Q\ll Q_0}\frac{1}{Q^2}\sup_{M_0\ll \frac{N^{2+\varepsilon}}{Q_0^3}}M_0^{1/3}\times \Delta,$$
where
\begin{equation}\label{largesievedualityDelta}
\begin{split}
\Delta=\underset{||\alpha||^2=1}{\sup\limits_{\alpha}}\underset{m_1^2m_2\sim M_0}{\sum_{m_1>0}\sum_{m_2>0}}\bigg|&\int\limits_{\xi\sim X}\sum_{\substack{q\sim Q\\m_1|q}}\sum_{h\sim \frac{qT}{N}}\alpha(\xi,q,h)S(\bar h,m_2;qm_1^{-1})m_1\\&\times e\bigg(\frac{3(T+r\xi)^{1/3} m_1^{2/3}m_2^{1/3}}{(2\pi)^{1/3}q^{2/3}h^{1/3}}\bigg)d\xi\bigg|^2.
\end{split}
\end{equation}
where $\alpha(\xi,q,h)$ varies over all the complex vectors such that $\int_{\xi\sim X}\sum_{q\sim Q}\sum_{h\sim \frac{qT}{N}}|\alpha|^2=1$. 
\end{lemma}	

\subsection{After large sieve}
At this point, we aim to achieve a non-trivial upper bound for $\Delta$. First of all, as $m_1|q$, everywhere in the expression, we will replace $q$ by $m_1q$ where the new $q$ is in the range $q\sim Q/m_1$. Now, we open up the absolute square in (\ref{largesievedualityDelta}) and denote the $m_2$ sum as $S_{\cM_2}$. Thus, we get 
\begin{equation}
\begin{split}
\Delta=\sup_{||\alpha||_2=1}&\sum_{m_1\ll Q}m_1^2\int_{\xi_1\sim X}\sum_{q_1\sim Q/m_1}\sum_{h_1\sim\frac{QT}{N}}\int_{\xi_2\sim X}\sum_{q_2\sim Q/m_1}\sum_{h_2\sim\frac{QT}{N}}\alpha_1\bar\alpha_2\times S_{\cM_2},\\  \text{where }S_{\cM_2}=&\sum_{m_2\sim\frac{M_0}{m_1^2}}S(\bar h_1,m_2;q_1)S(\bar h_2,m_2;q_2)\\&\times  e\bigg(\frac{3(T+r\xi_2)^{1/3} (m_1^2m_2)^{1/3}}{(2\pi)^{1/3}(q_2m_1)^{2/3}h_2^{1/3}}-\frac{3(T+r\xi_1)^{1/3} (m_1^2m_2)^{1/3}}{(2\pi)^{1/3}(q_1m_1)^{2/3}h_1^{1/3}}\bigg).\label{SM2firstdef}
\end{split}
\end{equation}
We note that, for $i=1,2$, $\alpha(\xi_i,q_im_1, h_i)$ is written as $\alpha_i$ in short. Let us define $$\Xi(\xi,q,h)=\frac{(T+r\xi)^{1/3} }{(2\pi)^{1/3}q^{2/3}h^{1/3}}$$ and for brevity, we will often write $\Xi_1=\Xi(\xi_1,q_1m_1,h_1)$ or $\Xi_2=\Xi(\xi_2,q_2m_2,h_2)$.We note that 
$$|\Xi_2-\Xi_1|\ll\frac{T^{1/3}}{Q^{2/3}\cdot (QT/N)^{1/3}}\sim \frac{N^{1/3}}{Q}.$$
Then, we introduce a smooth dyadic partition with smooth functions of the form $V\left(\frac{\Xi_2-\Xi_1}{X_1}\right)$ where $T^{-A}\ll |X_1|\ll \frac{N^{1/3}}{Q}$. So, 
\begin{align}\Delta\ll T^{\varepsilon}\sup_{||\alpha||_2=1}\sup_{\frac{1}{T^A}\ll X_1\ll \frac{N^{1/3}}{Q}}\Bigg|&\sum_{m_1\ll Q}m_1^2\int\limits_{\xi_1\sim X}\sum_{q_1\sim \frac{Q}{m_1}}\sum_{h_1\sim\frac{QT}{N}}\alpha_1\nonumber \\&\int\limits_{\xi_2\sim X}\sum_{q_2\sim \frac{Q}{m_1}}\sum_{h_2\sim\frac{QT}{N}}\bar\alpha_2 V\left(\frac{\Xi_2-\Xi_1}{X_1}\right) S_{\cM_2}\Bigg|.\label{DeltaFirstModification}\end{align}
Here, we only deal with $X_1>0$ as the case of the negative $X_1$ is symmetrical and the contribution from $|\Xi_2-\Xi_1|\ll T^{-A}$ is negligibly small. 
\subsection{Poisson summation formula on the $m$ sum}
Now, we intend to apply the Poisson Summation formula on $S_{\cM_2}$ (\ref{SM2firstdef}). We note that $$S(\bar h_1,m_2;q_1)S(\bar h_2,m_2;q_2)$$
is of modulus $q_1q_2$. Applying Poisson summation formula (\ref{Poissonsum}) on the $S_{\cM_2}$ sum, we get 
\begin{align*}
    S_{\cM_2}=\frac{1}{q_1q_2}\sum_{m\in \mathbb{Z}}&\left(\sum_{\beta\bmod q_1q_2}S(\bar h_1,\beta;q_1)S(\bar h_2,\beta;q_2)e_{q_1q_2}(m\beta)\right)\\
   & \times \int_{y\sim \frac{M_0}{m_1^2}}e\left(3(\Xi_2-\Xi_1)m_1^{2/3}y^{1/3}\right)e(-my/q_1q_2)dy.
\end{align*}
We make a change of variable $y\mapsto \frac{M_0y}{m_1^2}$ and arrive at 
\begin{equation}\label{CharacterSum}
\begin{split}
S_{\cM_2}=&\frac{M_0}{m_1^2q_1q_2}\sum_{m\in\mathbb{Z}}\mathfrak{C}\ \mathfrak{J},\\
\text{where}~~~& \\
\mathfrak{C}=&\sum_{\beta\bmod q_1q_2}S(\bar h_1,\beta;q_1)S(\bar h_2,\beta;q_2)e_{q_1q_2}(m\beta),\\
\mathfrak{J}=&\int_{y\sim 1}e\left(3(\Xi_2-\Xi_1)M_0^{1/3}y^{1/3}\right)\times e\bigg(-\frac{mM_0y}{m_1^2q_1q_2}\bigg)dy.
\end{split}
\end{equation}
Depending on $m$, we will divide the analysis of the above expression into two separate cases: ``Diagonal" when $m=0$ and ``Off-Diagonal" when $m\neq0$. 
\section{Diagonal}\label{S9}
In this case, we have $m=0$. At first, we will calculate the contribution of the character sum, which is 
\begin{equation}\label{dia}
\begin{split}
\mathfrak{C}&=\sum_{\beta\bmod q_1q_2}S(\bar h_1,\beta;q_1)S(\bar h_2,\beta;q_2)=q_1q_2\underset{xq_2+yq_1\equiv 0\bmod q_1q_2}{\sideset{}{^*}\sum_{x\bmod q_1}\ \sideset{}{^*}\sum_{y\bmod q_2}}e\bigg(\frac{\bar{h_1}\bar{x}}{q_1}\bigg)e\bigg(\frac{\bar{h_2}\bar{y}}{q_2}\bigg).
\end{split}
\end{equation}
In order to have a solution, we must have $q_2|yq_1$, which would imply $q_2|q_1$. Similarly, we get that $q_1|q_2$. So we would have $q_1=q_2$. In that case, the character sum becomes
\begin{equation*}\label{key}
\mathfrak{C}=q_1^2\sideset{}{^*}\sum_{x\bmod q_1}e\bigg(\frac{(h_1-h_2)x}{q_1}\bigg)=q_1^2\mathfrak {c}_{q_1}(h_1-h_2)\delta_{q_1=q_2},
\end{equation*} 
where $\mathfrak c_{q_1}$ is the Ramanujan sum modulo $q_1$. Only when $h_1\equiv h_2\bmod{q_1}$, we have $\mathfrak{c}_{q_1}({h_1}-{h_2})\ll q_1$. In any other case, we may write $\mathfrak{c}_{q_1}({h_1}-{h_2})\ll (h_1-h_2,q_1)$. From the above arguments on the character sum, we can deduce that
\[\sum_{h_1}c_{q_1}({h_1}-{h_2})\ll \begin{cases}
&\ll \frac{Q^{1+\varepsilon}}{m_1}\text{ when }m_1\ll T^{1/2}\\
&\ll \frac{Q^{1+\varepsilon}T}{N}\text{ when } T^{1/2}\ll m_1\ll Q
\end{cases}\]

We recall $\Xi_k=\frac{(T+r\xi_k)^{1/3} }{(2\pi)^{1/3}q_k^{2/3}h_k^{1/3}}$ for $k=1,2$. So, we note that the integral \begin{equation*}\label{key}
\mathfrak{J}=\int_{y\sim 1}e\left(3(\Xi_2-\Xi_1)M_0^{1/3}y^{1/3}\right) dy
\end{equation*}
is negligible (by repeated integration by parts (Lemma \ref{repeatedint})) unless
\begin{equation*}\label{key}
\begin{split}
&|\Xi_2-\Xi_1|\ll \frac{T^{\varepsilon}}{M_0^{1/3}}
\implies \xi_2-\tilde{\xi_1}\ll\frac{QX}{M_0^{1/3}N^{1/3}}
\end{split}
\end{equation*}
where $\tilde \xi_1$ depends on $\xi_1, h_2$ and $h_2'$. 
Then, by Cauchy's inequality, $\Delta$ is bounded by 
\begin{equation*}\label{key}
\begin{split}
\Delta_{m=0}\ll & M_0 \bigg(\int_{\xi_1}\sum_{m_1\ll Q}\sum_{q_1\sim Q/m_1}\sum_{h_1\sim\frac{QT}{N}}|\alpha_1|^2\sum_{h_2\sim\frac{QT}{N}}\int\limits_{\xi_2-\tilde \xi_1\ll\frac{QX}{M_0^{1/3}N^{1/3}}}c_{q_1}({h_1}-{h_2})\bigg)^{1/2}\\
&\times \bigg(\int_{\xi_2}\sum_{m_1\ll Q}\sum_{q_2\sim Q/m_1}\sum_{h_2\sim\frac{QT}{N}}|\alpha_2|^2\sum_{h_1\sim\frac{QT}{N}}\int\limits_{\xi_1-\tilde \xi_2\ll\frac{QX}{M_0^{1/3}N^{1/3}}}c_{q_2}({h_1}-{h_2})\bigg)^{1/2}.
\end{split}
\end{equation*}
By our choice of $\alpha_1$ ($||\alpha_1||_2=1$), we note that 
\begin{align}
\int_{\xi_1}\sum_{m_1\ll Q}\sum_{q_1\sim Q/m_1}\sum_{h_1}|\alpha(q_1m_1,h_1,\xi_1)|^2\ll Q^{\varepsilon}.
\end{align}
Thus we get 
\begin{equation}\label{Deltam=0}
\Delta_{m=0}\ll M_0T^{\varepsilon}\cdot \frac{QX}{M_0^{1/3}N^{1/3}}\cdot \left(Q+\frac{QT}{N}\right)\ll \frac{Q^2XM_0^{2/3}T^{\varepsilon}}{N^{1/3}}
.\end{equation}
Hence, by Lemma \ref{Largesievemain}, the contribution of the diagonal part in $\mathfrak{M}$ is 
\begin{align}
    \mathfrak{M}\ll &\frac{N^{4/3}}{T}\sup_{Q\ll Q_0}\frac{1}{Q^2}\sup_{M_0\ll \frac{N^{2+\varepsilon}}{Q_0^3}}M_0^{1/3}\times \frac{Q^2XM_0^{2/3}T^{\varepsilon}}{N^{1/3}}\ll\frac{N^3X}{TQ_0^3}\asymp \frac{N^{3/2}\sqrt{T}}{\sqrt{X}}.\label{mathfrakmdia}
\end{align}
\begin{remark}We require, $\mathfrak{M}\ll NX\sqrt{T}$. Hence, the bound in the diagonal part is smaller than the expected bound as long as $X\gg T^{1/2}$. \end{remark}
\section{Off-Diagonal}\label{S10}
For the $m\neq0$ case, we will start with the analysis of Oscillatory integral $\mathfrak{J}$ in (\ref{CharacterSum}). We recall 
\begin{align}
    \mathfrak{J}=&\int_{y\sim 1}e\left(3(\Xi_2-\Xi_1)M_0^{1/3}y^{1/3}\right)\times e\bigg(-\frac{mM_0y}{m_1^2q_1q_2}\bigg)dy\label{10.0.1}
\end{align}
and $\Xi_2-\Xi_1\sim X_1$. We will now assume that $X_1\gg \frac{T^{\varepsilon}}{M_0^{1/3}}$.  

When $X_1\ll \frac{T^{\varepsilon}}{M_0^{1/3}}$, we can absorb the integral $\mathfrak J$ into the smooth functions if $|m|\ll \frac{Q^2T^{\varepsilon}}{M_0}$ and $\mathfrak{J}$ would be negligibly small for the other range of $m$. This case would be considered separately in Section \ref{X_1small}.  

\subsection{Integral $\mathfrak{J}$}

\begin{lemma}\label{Jintegrallemma}	Let $X_1\gg \frac{T^{\varepsilon}}{M_0^{1/3}}$. Then, if $m\sim\frac{Q^2X_1}{M_0^{2/3}}$, we have  \begin{equation*}\begin{split}
    \mathfrak{J}\asymp\frac{1}{\sqrt{X_1M_0^{1/3}}}\times e\bigg(2\sqrt{\frac{q_1q_2m_1^2}{m}}(\Xi_2-\Xi_1)^{3/2}\bigg)+O(T^{-A})
    \end{split}
.\end{equation*}
In the rest of the range of $m$,  $\mathfrak{J}$ is negligibly small. \end{lemma} 
\begin{proof}
The exponential integral $\mathfrak{J}$ (\ref{10.0.1}) is of the form $\mathfrak{J}=\int_{y\sim 1}e(f(y))$ where $f(y)=3Ay^{1/3}-By$, where 
\begin{align*}
    A=(\Xi_2-\Xi_1)M_0^{1/3}\text{ and }B=\frac{mM_0}{m_1^2q_1q_2}.
\end{align*}
The first and the second order derivatives of $f(y)$ are 
\begin{align*}
    f'(y)&=\frac{A}{y^{2/3}}-B~~\text{ and }~~
    f''(y)=\frac{2A}{3y^{5/3}}
\end{align*}
As $X_1\gg \frac{T^{\varepsilon}}{M_0^{1/3}}$, there would be a stationary point if and only if  $m\sim \frac{Q^2X_1}{M_0^{2/3}}$. If there is no stationary point, by repeated integration by parts (Lemma \ref{repeatedint}), we note that the integral is negligibly small. Then, by Lemma \ref{StationaryPhase1}, we conclude the proof. 
\end{proof}
\subsection{Preliminary Analysis of Character Sum}\label{S10.2}
We will simplify the character sum $\mathfrak{C}$ in (\ref{CharacterSum}) to make it conducive for the application of the Poisson Summation formula. For that, we will introduce a new set of notations that will be used in the rest of this paper.  \\\\
\textbf{Notation:} Let $d=(q_1,q_2)$. Then let $q_1=u_1v_1$ and $q_2=u_2v_2$ such that $v_1,v_2|d^{\infty}$ and $u_1,u_2,d$ are mutually coprime, i.e. $v_1$ is the part of $q_1$ corresponding to the prime divisors of $d$ and $u_1$ is the rest of $q_1$ which is consequently coprime to $d$ . Hence, $u_1,u_2$ and $v_1v_2$ are mutually coprime. We further denote $\tilde v_1=v_1/d$ and $\tilde v_2=v_2/d$. Later in this section, we will observe that $d|m$. Thus we denote $\tilde{m}=m/d$.

Now, we may further split $d=(q_1,q_2)$ into $d=d_0\cdot d_1\cdot d_2$ in the following manner:
$$d_1=(d, \tilde v_1^{\infty}),\ d_2=(d,\tilde v_2^\infty)\text{ and }d_0=\frac{d}{d_1d_2}. $$
We note that $(d_0,\tilde v_1\tilde v_2)=1$. We will also observe that $(\tilde m,\tilde v_1\tilde v_2)=1$. Then, $(\tilde m, d)|d_0$. We recall that $\tilde m=m/d$. Now, we denote $d_m=(\tilde m,d_0^{\infty})$ and $m^*=\frac{m}{dd_m}$. In a nutshell 
$$q_1=d\tilde v_1u_1,~q_2=d\tilde v_2u_2, ~d=d_0d_1d_2,~m=\tilde md=m^*dd_m.$$
We break up the character sum $\mathfrak{C}$
$$\mathfrak{C}=\sum_{\beta\bmod q_1q_2}S(\bar h_1,\beta;q_1)S(\bar h_2,\beta;q_2)e_{q_1q_2}(m\beta)$$
into character sums of mutually coprime modulus $u_1$, $u_2$ and $v_1v_2$. So, we write $$\mathfrak{C}=q_1q_2\mathfrak{C}_1\mathfrak{C}_2\mathfrak{C}_3,$$
where 
\begin{align}
    \mathfrak{C}_1&=\frac{1}{u_1}\sum_{\beta_1\bmod u_1}S(\bar{h_1}\bar{v_1},\beta_1\bar{v_1};u_1)e_{u_1}(m\bar{u_2}\overline{v_1v_2}\beta_1)=e\bigg(-\frac{\bar{m}\bar{h_1}\bar{v_1}u_2v_2}{u_1}\bigg),\\
\mathfrak{C}_2&=\frac{1}{u_2}\sum_{\beta_2\bmod u_2}S(\bar{h_2}\bar{v_2},\beta_2\bar{v_2};u_2)e_{u_2}(m\bar{u_1}\overline{v_1v_2}\beta_2)=e\bigg(-\frac{\bar{m}\bar{h_2}\bar{v_2}u_1v_1}{u_2}\bigg),\\
\mathfrak{C}_3&=\frac{1}{v_1v_2}\sum_{\beta_3\bmod v_1v_2}S(\bar{h_1}\bar{u_1},\beta_3\bar{u_1};v_1)S(\bar{h_2}\bar{u_2},\beta_3\bar{u_2};v_2)e_{v_1v_2}(m\bar{u_1}\bar{u_2}\beta_3)\nonumber\\
&=\underset{u_2v_2x_1+u_1v_1x_2\equiv -m\bmod v_1v_2}{\sideset{}{^*}\sum_{x_1\bmod v_1}\sideset{}{^*}\sum_{x_2\bmod v_2}}e\bigg(\frac{\bar{h_1}\bar{u_1}\bar{x_1}}{v_1}\bigg)e\bigg(\frac{\bar{h_2}\bar{u_2}\bar{x_2}}{v_2}\bigg).
\end{align}
We also get $(m,u_1u_2)=1$. By reciprocity, we can write $\mathfrak{C}_1$ as 
$$\mathfrak{C}_1=e\bigg(-\frac{\bar{m}\bar{h_1}\bar{v_1}u_2v_2}{u_1}\bigg)=e\bigg(\frac{\bar{u_1}u_2\tilde v_2}{mh_1\tilde v_1}\bigg)\times e\bigg(-\frac{u_2\tilde v_2}{u_1mh_1\tilde v_1}\bigg)$$
and absorb the second term into the smooth functions. We treat $\mathfrak{C}_2$ in a similar manner. We also note that the congruence relation $u_2v_2x_1+u_1v_1x_2\equiv -m\bmod v_1v_2$ implies $d|m$ and $(\tilde m, \tilde v_1\tilde v_2)$. Then, we can write $\mathfrak{C}$ as $\mathfrak{C}=q_1q_2\cdot\fC_0 \cdot V(\cdots)$ where 
\begin{align}
     \fC_0=e\left(\frac{\tilde v_2u_2\overline{u_1}}{mh_1\tilde v_1}\right)e\left(\frac{\tilde v_1u_1\overline{u_2}}{mh_2\tilde v_2}\right)\underset{x_1\tilde v_2+x_2\tilde v_1+\tilde m\equiv0 \bmod d\tilde v_1\tilde v_2}{\sumstar_{x_1\bmod d\tilde v_1}\sumstar_{x_2\bmod d\tilde v_2}}e\left(\frac{\overline x_1\overline h_1u_2\overline {u_1}}{d\tilde v_1}\right)e\left(\frac{\overline x_2 \overline h_2u_1\overline{u_2}}{d\tilde v_2}\right).\label{mathfrakc0}
 \end{align}
 Here, we have used the notation $\tilde v_i=v_i/d$. We may further simplify the character sum $\cC$ and get 
 \begin{lemma}\label{CharsumSimplifiedA1A2}
     \begin{align}
    \fC_0=\sumstar_{x_0\bmod d_0}~\sumstar_{x_1\bmod d_1}~\sumstar_{x_2\bmod d_2}e\left(\frac{\overline{u_1}u_2A_1}{mh_1}\right)e\left(\frac{u_1\overline{u_2}A_2}{mh_2}\right)
\end{align}
where 
\begin{align}
    A_1\equiv \begin{cases}
        \tilde v_2\overline{\tilde v_1}&\pmod {m^* h_1},\\
        \tilde v_2\overline{\tilde v_1}[\overline{1+\bar x_0d_m}]&\pmod {d_0d_m},\\
         \tilde v_2(\overline{\bar x_1+\tilde v_1})& \pmod {d_1},\\
        \overline{\tilde v_1x_2}&\pmod {d_2},
    \end{cases}~\text{and}~
     A_2\equiv \begin{cases}
        \tilde v_1\overline{\tilde v_2}&\pmod {m^* h_2},\\
        \tilde v_1\overline{\tilde v_2}[1+\bar x_0d_m]&\pmod {d_0d_m},\\
        \overline{\tilde v_2x_1}&\pmod {d_1},\\
        \tilde v_1(\overline{\bar x_2+\tilde v_2})& \pmod {d_2}\label{A2charsum}
    \end{cases}.
\end{align}
 \end{lemma}
\begin{proof}
     We want to split the character sum $\mathfrak{C}_0$ (\ref{mathfrakc0}) into character sums of the following four coprime moduli ($m^* h_1h_2,~ d_0d_m,~ d_1\tilde v_1$ and $d_2\tilde v_2$). Firstly,  by the  Chinese Remainder Theorem, we split the $x_1$ and the $x_2$ sum:
     $$\sumstar_{x_1\bmod d\tilde v_1}\mapsto \sumstar_{x_{1,0}\bmod d_0}\sumstar_{x_{1,1}\bmod d_1\tilde v_1}\sumstar_{x_{1,2}\bmod d_2}\text{ and }\sumstar_{x_2\bmod d\tilde v_2}\mapsto \sumstar_{x_{2,0}\bmod d_0}\sumstar_{x_{2,1}\bmod d_1}\sumstar_{x_{2,2}\bmod d_2\tilde v_2}.$$
     With these notations, we can also split the congruence relation $$x_1\tilde v_2+x_2\tilde v_1+\tilde m\equiv0 \bmod d\tilde v_1\tilde v_2$$
into 
\begin{align}
    &x_{1,0}\tilde v_2+x_{2,0}\tilde v _1+\tilde m\equiv 0 \bmod{d_0}\iff x_{1,0}\equiv-\overline{\tilde v_2}(x_{2,0}\tilde v _1+\tilde m)\bmod{d_0}\\
    &x_{1,1}\tilde v_2+x_{2,1}\tilde v_1+\tilde m\equiv 0 \bmod {d_1\tilde v_1}\iff\overline{x_{1,1}}\equiv -\tilde v_2\overline{(x_{2,1}\tilde v_1+\tilde m)}\bmod {d_1\tilde v_1}\\
    &x_{1,2}\tilde v_2+x_{2,2}\tilde v_1+\tilde m\equiv 0 \bmod{d_2\tilde v_2}\iff\overline{x_{2,2}}\equiv -\tilde v_1\overline{(x_{1,2}\tilde v_2+\tilde m)}\bmod {d_2\tilde v_2}
\end{align}
Now, we may split the character sum into the following four clusters of coprime moduli ($m^* h_1h_2,~ d_0d_m,~ d_1\tilde v_1$ and $d_2\tilde v_2$):
\begin{align}
    \mathfrak{C}_0=\fC_h\cdot \fC_m\cdot \fC_1\cdot \fC_2
\end{align}
where $\fC_h,~ \fC_m,~ \fC_1$ and $ \fC_2$ are defined below.
\begin{align}
    \fC_h=& e\left(\frac{u_2\overline{u_1}\cdot\tilde v_2 \overline{d_m d\tilde v_1}}{m^* h_1}\right)e\left(\frac{u_1 \overline{u_2}\cdot \tilde v_1 \overline{d_md\tilde v_2}}{m^* h_2}\right)=e\left(\frac{u_2\overline{u_1}\cdot\tilde v_2\overline{\tilde v_1} \cdot \overline{\frac{mh_1}{m^*h_1}}}{m^* h_1}\right)e\left(\frac{u_1 \overline{u_2}\cdot \tilde v_1 \overline{\tilde v_2}\cdot \overline{\frac{mh_2}{m^*h_2}}}{m^* h_2}\right).
\end{align}
\begin{align*}
    \fC_m=& e\left(\frac{u_2 \overline{u_1} \tilde v_2\overline{\frac{mh_1\tilde v_1}{d_0d_m}}}{d_0d_m}\right)e\left(\frac{u_1\overline{u_2}\tilde v_1\overline{\frac{mh_2\tilde v_2}{d_0d_m}}}{d_0d_m}\right)\\&\underset{x_{1,0}\equiv-\overline{\tilde v_2}(x_{2,0}\tilde v _1+\tilde m)\bmod{d_0}}{\sumstar_{x_{1,0}\bmod d_0}~\sumstar_{x_{2,0}\bmod d_0}} 
    e\left(\frac{\overline x_{1,0}\overline h_1u_2\overline {u_1}\overline{\frac{d\tilde v_1}{d_0}}}{d_0}\right)e\left(\frac{\overline x_{2,0} \overline h_2u_1\overline{u_2}\overline{\frac{d\tilde v_2}{d_0}}}{d_0}\right)\\
    =& \sumstar_{x_{0}\bmod d_0}e\left(\frac{u_2 \overline{u_1} \tilde v_2\overline{\tilde v_1}\cdot (\overline{1+\bar x_0d_m})\cdot \overline{\frac{mh_1}{d_0d_m}}}{d_0d_m}\right)e\left(\frac{u_1\overline{u_2}\tilde v_1\overline{\tilde v_2}\cdot (1+\bar x_0d_m)\overline{\frac{mh_2}{d_0d_m}}}{d_0d_m}\right).
\end{align*}
Here, we have made a change of variable $x_0m^*\equiv x_{2,0}\tilde v_1\bmod d_0$. We recall the notation $m=d\tilde m=dd_m\tilde m^*$. 

\begin{align*}
    \fC_1=&e\left(\frac{u_2\overline{u_1}\tilde v_2\overline{\frac{mh_1}{d_1}}}{d_1\tilde v_1}\right)\underset{\overline{x_{1,1}}\equiv -\tilde v_2\overline{(x_{2,1}\tilde v_1+\tilde m)}\bmod {d_1\tilde v_1}}{\sumstar_{x_{1,1}\bmod d_1\tilde v_1}\sumstar_{x_{2,1}\bmod d_1}}e\left(\frac{u_2\overline {u_1} \overline x_{1,1}\overline h_1\cdot \overline{\frac{d}{d_1}}}{d_1\tilde v_1}\right)e\left(\frac{u_1\overline{u_2} \overline x_{2,1} \overline h_2 \overline{\frac{\tilde v_2d}{d_1}}}{d_1}\right)\\
    =& \sumstar_{x_{1}\bmod d_1}e\left(\frac{u_2\overline {u_1}\tilde v_2(\overline{\bar x_1+\tilde v_1})\cdot \overline{\frac{mh_1}{d_1}}}{d_1}\right)e\left(\frac{u_1\overline{u_2}\overline{\tilde v_2}\cdot \overline x_{1} \overline{\frac{mh_2}{d_1}}}{d_1}\right).
\end{align*}
Here, we changed the variable $x_{2,1}$ to $x_{1}$, where $x_{2,1}\equiv x_1\tilde m\bmod d_1$. 
Similarly, we derive
\begin{align*}  \fC_2=& e\left(\frac{u_2\overline{u_1}\overline{\tilde v_1}\cdot \overline x_{2} \cdot \overline{\frac{mh_1}{d_2}}}{d_2}\right)\sumstar_{x_{2}\bmod d_2}e\left(\frac{u_1\overline {u_2}\tilde v_1(\overline{\bar x_2+\tilde v_2})\cdot \overline{\frac{mh_2}{d_2}}}{d_2}\right).
\end{align*}
Now, by the Chinese remainder theorem, we derive the statement of the lemma. 
\end{proof}

Now, we denote $H_1:=\frac{qT}{N}$ and $M_1:=\frac{Q^2X_1}{M_0^{2/3}}$ which are respectively the dyadic range of the $h$ and the $m$ sum. Thus, after the Poisson summation formula on the $m$ sum, (\ref{DeltaFirstModification}) transforms into 
\begin{align}\Delta\ll T^{\varepsilon} \sup_{||\alpha||_2=1}\sup_{X_1\ll \frac{N^{1/3}}{Q}}\Bigg|&\sum_{m_1\ll Q}m_1^2\int\limits_{\xi_1\sim X}\sum_{q_1\sim \frac{Q}{m_1}}\sum_{h_1\sim H_1}\nonumber\\& \int\limits_{\xi_2\sim X}\sum_{q_2\sim \frac{Q}{m_1}}\sum_{h_2\sim H_1}\alpha_1\bar\alpha_2 V\left(\frac{\Xi_2-\Xi_1}{X_1}\right) S_{\cM_2}\Bigg|,\label{DeltaAfterfirstpoisson}\end{align}
where, for $m\neq 0$ and $X_1\gg \frac{T^{\varepsilon}}{M_0^{1/3}}$, by Lemma \ref{Jintegrallemma} and Lemma \ref{CharsumSimplifiedA1A2}, we have 
\begin{align}
    S_{\cM_2}
    =&\frac{M_0}{m_1^2}\cdot \sum_{m\sim M_1} \frac{1}{\sqrt{X_1M_0^{1/3}}}\cdot \mathfrak{C}_0 \cdot e\bigg(2\sqrt{\frac{q_1q_2m_1^2}{m}}(\Xi_{2}-\Xi_1 )^{3/2}\bigg)+O(T^{-A}),
\end{align}
where $\Xi_i=\Xi(\xi_i,m_1q_i,h_i)=\frac{(T+r\xi_i)^{1/3} }{(2\pi)^{1/3}(m_1q_i)^{2/3}h_i^{1/3}}$ for $i=1,2$. 
\subsection{Cauchy's inequality and Poisson Summation Formula on $q_1$}
We now come back to (\ref{DeltaFirstModification}) and modify the  sum over $m,q_1$ and $q_2$ into 
\begin{equation*}
\sum_{m\sim M_1}\sum_{q_1\sim Q/m_1}\sum_{q_2\sim Q/m_1}\rightsquigarrow \sum_{d\ll M_1}\underset{(\tilde{m},u_1u_2\tilde v_1\tilde v_2)=1}{\sum_{\tilde{m}\sim\frac{M_1}{d}}}\underset{(\tilde v_1,\tilde v_2)=1}{\sum_{\tilde v_1,\tilde v_2|d^{\infty}}}\sum_{u_1\sim \frac{Q}{d\tilde v_1m_1}}\sum_{u_2\sim \frac{Q}{d\tilde v_2m_1}}
,\end{equation*} 
which is just a simple interpretation of the notations defined in Section \ref{S10.2}.  We then apply Cauchy's inequality on all the sums except $\xi_2,~h_2$ and $u_2$ in (\ref{DeltaAfterfirstpoisson}). When $X_1\gg \frac{T^{\varepsilon}}{M_0^{1/3}}$, we get 
\begin{equation}\label{18.7}
    \begin{split}
        \Delta\ll &\sup_{||\alpha||_2=1}\sup_{\frac{T^{\varepsilon}}{M_0^{1/3}}\ll X_1\ll \frac{N^{1/3}}{Q}}Q\sqrt{M_0}\Bigg(\sum_{m_1\ll Q}\sum_{d\ll M_1}\frac{1}{d}\underset{(\tilde v_1,\tilde v_2)=1}{\sum_{\tilde v_1,\tilde v_2|d^{\infty}}}\int_{\xi_1\sim X}\sum_{h_1\sim H_1}\sum_{\tilde{m}\sim M_1/d}
        \\&\sum_{u_1\sim \frac{Q}{d\tilde v_1m_1}}\Bigg|\int_{\xi_2\sim X}\sum_{h_2\sim H_1}\sum_{u_2\sim \frac{Q}{d\tilde v_2m_1}}\alpha_2 \mathfrak{C}_0e\bigg(2\sqrt{\frac{q_1q_2m_1^2}{m}}(\Xi_{2}-\Xi_1 )^{3/2}\bigg)V\left(\frac{\Xi_2-\Xi_1}{X_1}\right)\Bigg|^2\Bigg)^{1/2}.
    \end{split}
\end{equation}
Here, we have used the fact that 
\begin{align}
    \sum_{m_1}\sum_{d}\underset{(\tilde v_1,\tilde v_2)=1}{\sum_{\tilde v_1,\tilde v_2|d^{\infty}}}\sum_{u_i\sim \frac{Q}{d\tilde v_im_1}}\int_{\xi_i\sim X}\sum_{h_i\sim\frac{QT}{N}}|\alpha(m_1d\tilde v_iu_i,\xi_i,h_i)|^2\ll T^{\varepsilon}.\label{vectorl2norm}
\end{align}
We open up the absolute square in (\ref{18.7}), denote all the copies of variables by $'$ in superscript (after opening the absolute square, variable $u_2$ has two copies. We denote one by $u_2$ and another one by $u_2'$) and the $u_1$ sum is of the form 
\begin{equation}\label{Firstdefofce}
\begin{split}
\mathcal{E}:=&\sum_{u_1\sim \frac{Q}{d\tilde v_1m_1}} \mathfrak{C}_0\overline{\mathfrak{C}_0'}e\bigg(2\sqrt{\frac{\tilde v_1du_1q_2m_1^2}{m}}(\Xi_{2}-\Xi_1 )^{3/2}-2\sqrt{\frac{\tilde v_1du_1q_2'm_1^2}{m}}(\Xi_{2}'-\Xi_1 )^{3/2}\bigg)
\end{split}          
\end{equation}
where $q_2=d\cdot \tilde v_2\cdot u_2$, $q_2'=d\cdot \tilde v_2\cdot u_2'$, $\Xi_2'=\Xi(\xi_2',q_2'm_1,h_2')$, $\Xi(\xi,q,h)=\frac{(T+r\xi)^{1/3} }{(2\pi)^{1/3}q^{2/3}h^{1/3}}$ and $\mathfrak{C}_0'$ is a copy of $\mathfrak{C}_0$, i.e. 
\begin{align*}
    \mathfrak{C}_0'=\sideset{}{^*}\sum_{x_0'\bmod d_0}~\sideset{}{^*}\sum_{x_1'\bmod d_1}~\sideset{}{^*}\sum_{x_2'\bmod d_2}~e\left(\frac{\bar u_1u_2' A_1}{mh_1}+\frac{\bar u_2'u_1A_2'}{mh_2'}\right).
\end{align*}
Now, we apply the Poisson summation formula (\ref{Poissonsum}) to the $u_1$ sum (\ref{Firstdefofce}). We note that $\fC_0\bar \fC_0'$ is of modulus $mh_1h_2h_2'$, which  
gives a character sum $\mathcal{C}$ modulo $mh_1h_2h_2'$, and the rest of the $\mathcal{E}$ gives us an oscillatory integral (see \ref{Poissonsum}).  In the oscillatory integral,  we make a change of variable $y\mapsto \frac{yQ}{m_1\tilde v_1d}$. We also bring the $\xi_1$ integral inside to get a two-variable oscillatory integral $\mathcal{I}$. Once we write out all the notations in their original form, we get 
\begin{lemma}
\begin{equation}\label{afterpoissonofq1}
\int_{\xi_1\sim X}\mathcal{E}\ \ d\xi_1=\frac{Q}{m_1v_1mh_1h_2h_2'}\ \sum_{\tilde{u_1}\in \mathbb{Z}}\ \mathcal{C}\ \mathcal{I}	,
\end{equation}
where 
\begin{equation}\label{afterpoissonofq1detailed}
\begin{split}
\mathcal{C}=&\sideset{}{^*}\sum_{x_0\bmod d_0}~\sideset{}{^*}\sum_{x_1\bmod d_1}~\sideset{}{^*}\sum_{x_2\bmod d_2}~\sideset{}{^*}\sum_{x_0'\bmod d_0}~\sideset{}{^*}\sum_{x_1'\bmod d_1}~\sideset{}{^*}\sum_{x_2'\bmod d_2}\\
&\times \underset{(\alpha,mh_1)=1}{\sum_{\alpha\bmod mh_1h_2h_2'}}e\left(\frac{\bar \alpha u_2 A_1}{mh_1}+\frac{\bar u_2\alpha A_2}{mh_2}\right)e\left(-\frac{\bar \alpha u_2' A_1'}{mh_1}-\frac{\bar u_2'\alpha A_2'}{mh_2'}\right)e\bigg(\frac{\tilde{u_1}\alpha}{mh_1h_2h_2'}\bigg)
\end{split}\end{equation}
and 
\begin{equation}\label{afterpoissonofq1detailedI}
\begin{split}
\mathcal{I}	=\int_{\br}\int_{\br}e&\bigg(2\sqrt{\frac{Qyq_2m_1}{m}}\bigg(\frac{(T+r\xi_2)^{1/3} }{(2\pi q_2^2m_1^2h_2)^{1/3}}-\frac{(T+r\xi_1)^{1/3}}{(2\pi (Qy)^2h_1)^{1/3}}\bigg)^{3/2}\\&-2\sqrt{\frac{Qy q_2'm_1}{m}}\bigg(\frac{(T+r\xi_2')^{1/3} }{(2\pi q_2'^{2}m_1^2h_2')^{1/3}}-\frac{(T+r\xi_1)^{1/3}}{(2\pi (Qy)^{2}h_1)^{1/3}}\bigg)^{3/2}\bigg)\\\times e&\bigg(-\frac{\tilde{u_1}y Q}{dm_1\tilde v_1mh_1h_2h_2'}\bigg)V\left(\frac{\xi_1}{X}\right)V(y)V\left(\frac{\Xi_2-\Xi_1}{X_1}\right)V\left(\frac{\Xi_2'-\Xi_1}{X_1}\right)dyd\xi_1.\\
\end{split}
\end{equation}  
\end{lemma}
\subsection{Analysis of $\mathcal{C}$:}
We recall that $m=d\cdot d_m\cdot m^*$ where $d_m=(m/d,d^{\infty})$ and $(d_m,\tilde v_1\tilde v_2)=1$. We recall that, $(dd_m,m^*h_1h_2h_2')=1$ because $(q_1,h_1)=1$, $(q_2,h_2)=1$ and $(q_2',h_2')=1$. Now, we define a set of notations similar to the section \ref{S10.2} to group the prime factors of $m^*h_1h_2h_2'$ to maintain coprimality.

Let $\hat h_1=(h_1,(m^*h_2h_2')^{\infty})$ and $\tilde h_1=\frac{h_1}{\hat h_1}$. Let $\hat h_{2}=(h_2,(m^* h_1)^\infty)$, $\hat h'_{2}=(h_2',(m^*h_1)^\infty)$, $\tilde h_2=\frac{h_2}{\hat h_2}$ and $\tilde h_2'=\frac{h_2'}{\hat h_2'}$. Then, we may write 
$$mh_1h_2h_2'=(dd_m)\cdot (m^*\hat h_1\hat h_{2}\hat h'_{2})\cdot (\tilde h_1)\cdot (\tilde h_2\tilde h_2')$$
where every group inside the parentheses is mutually coprime. We further denote $\cD=dd_m$ and  $\cM=m^*\hat h_1\hat h_{2}\hat h'_{2}$.   If we recall the definition of $A_1$ and $A_2$ (Lemma \ref{CharsumSimplifiedA1A2}), we note that $A_1\equiv A_1'\bmod m^*h_1$ and none of $A_1,A_1',A_2,A_2'$ modulo any of  $\cM, \tilde h_1,\tilde h_2, \tilde h_2'$ depends on $x_i$ and $x_i'$. Then, by CRT, we can write 
\begin{align*}
    \cC=\cC_d\cdot \cC_m\cdot \cC_{h1}\cdot \cC_{h2}
\end{align*}
where 
\begin{align}
    \cC_d
    =& \sideset{}{^*}\sum_{x_0\bmod d_0}~\sideset{}{^*}\sum_{x_1\bmod d_1}~\sideset{}{^*}\sum_{x_2\bmod d_2}~\sideset{}{^*}\sum_{x_0'\bmod d_0}~\sideset{}{^*}\sum_{x_1'\bmod d_1}~\sideset{}{^*}\sum_{x_2'\bmod d_2}\nonumber\\
    & \times S( \overline{m^*h_1}(A_1u_2-A_1'u_2'),\bar u_2A_2\overline{m^*h_2}-\bar u_2'A_2'\overline{m^*h_2'}+\tilde u_1 \overline{m^*h_1h_2h_2'};dd_m),\label{Cd}\\
    \cC_m
    =& S( A_1 \overline{\cD\tilde h_1}\hat h_{2}\hat h'_{2} (u_2-u_2'),\overline{\cD\tilde h_2\tilde h_2'}\hat h_1(\bar u_2A_2h_2'-\bar u_2'A_2'h_2+\tilde u_1\overline{h_1}); m^*\hat h_1\hat h_2\hat h_2' ),\label{CM}\\
    \cC_{h1}%=&\sumstar_{\beta\bmod \tilde h_1}e_{\tilde h_1}(\bar \beta A_1\overline{\cD m^*\hat h_1}(u_2-u_2')+\beta \tilde u_1\overline{m^*\hat h_1h_2h_2'} )\\
    =&\begin{cases}S(A_1\overline{\cD m^*\hat h_1}(u_2-u_2'),\tilde u_1\overline{m^*\hat h_1h_2h_2'}; \tilde h_1)&\text{when }\tilde u_1\neq 0~\\
     \mathfrak{c}_{\tilde h_1}(u_2-u_2')& \text{when }\tilde u_1=0,\end{cases}\label{Ch1}\\\text{ and }~~&\nonumber\\
    \cC_{h2}
    =&\tilde h_2\tilde h_2'~~\delta(\bar u_2 A_2h_2'-\bar u_2'A_2' h_2+\tilde u_1\overline{h_1}\equiv 0 \bmod \tilde h_2\tilde h_2').\label{Ch2}
\end{align}
For $\cC_d$, we will always use the trivial bound for $\cC_d\ll d^3d_m$ and the Kloosterman sums, we will use the Weil's bound (see \cite[Corollary 11.12]{IK}) $S(a,b;c)\ll_{\varepsilon} c^{\varepsilon}(a,b,c)^{1/2}c^{1/2}$ and for the Ramanujan sum, we will use $\mathfrak{c}_q(a)\leq (a,q)$ (see \cite[(3.5)]{IK}) 

In particular, when $\tilde u_1\neq 0$, $u_3=u_2-u_2'\neq 0$, analyzing the congruence relations arising from $\cC_{h2}$ and from the Weil's bound of $\cC_m$, $\cC_{h1}$, we get 
\begin{align}
    \sum_{0<|u_3|\ll \frac{Q_2}{m_1\tilde v_2d}}|u_3|^{-1}\sum_{\tilde u_1\sim Q_1}|\cC|\ll  T^{\varepsilon}d^3d_m\sqrt{m^* h_1}h_2h_2'\left(\frac{Q_1\hat h_1}{h_2h_2'}+1\right).\label{charactersumgenericcase}
\end{align}
When $\tilde u_1=0$, we get $\tilde h_2=\tilde h_2'$. Then, we use the trivial bound on $\cC_d$, $\cC_{h2}$, Weil's bound on $\cC_m$ and the Ramanaujan sum bound on $\cC_{h1}$ to bound 
\begin{align}
    \sum_{|u_3|\ll \frac{Q_2}{m_1d\tilde v_2}}|\cC|\ll T^{\varepsilon}d^3d_m\sqrt{m^*\hat h_1}h_2h_2'\left( \frac{Q_2}{m_1d\tilde v_2}+\sqrt{m^*\hat  h_1}\tilde h_1\right).\label{charsumu1zero}
\end{align}
When $\tilde u_1\neq 0$ and $u_3$ is small, we may use a similar bound of the form
\begin{align}
    \sum_{|u_3|\ll \frac{Q_2}{m_1d\tilde v_2}}\sum_{0\neq|\tilde u_1|\ll Q_3}|\cC|\ll T^{\varepsilon}d^3d_m\sqrt{m^*\hat h_1} h_2 h_2'\cdot Q_3\cdot \left(\frac{Q_2\sqrt{\tilde h_1}}{m_1d\tilde v_2}+\sqrt{m^*\hat h_1}\right).\label{charsumu3smallu1nonzero}
\end{align}

\subsection{Analysis of $\mathcal{I}$}
We will now analyze the two variables oscillatory integral $\mathcal{I}$ (\ref{afterpoissonofq1detailedI}) by Lemma \ref{Hormandertheorem} and Lemma \ref{Hormanderlemma }. We recall that   $\Xi_i\sim \frac{T^{1/3}}{(Q^2\cdot \frac{QT}{N})^{1/3}}\sim \frac{N^{1/3}}{Q}$. So, let us denote $Y_1=\frac{X_1}{N^{1/3}/Q}\ll 1$. Due to the smooth functions $V((\Xi_2-\Xi_1)/X_1)$ and $V((\Xi_2'-\Xi_1)/X_1)$, the actual support of the variables $\xi_1$ and $y$ are respectively of the size  $XY_1$ and $Y_1$ instead of $X$ and $1$. So, let us make the following changes of variables
$$\xi_1\mapsto XY_1z_1~~~\text{ and }~~~~y\mapsto y_1Y_1, ~~~~~~ d\xi_1dy=XY_1^2 dz_1dy_1.$$
With these notations, let us also define the smooth function 
$$U(z_1,y_1):=V(Y_1z_1)V(Y_1y_1)V\left(\frac{\Xi_2-\Xi_1}{X_1}\right)V\left(\frac{\Xi_2'-\Xi_1}{X_1}\right).$$
$U(z_1,y_1)$ is supported in a compact set $K\subset [Y_1^{-1},2Y_1^{-1}]\times  [Y_1^{-1},2Y_1^{-1}]$ and satisfies 
\begin{align}
    \frac{\partial^i\partial^j}{\partial z_1^i\partial z_1^j}U(z_1,y_1)\ll \frac{1}{X_1^{i+j}}\frac{\partial^i\partial^j\Xi_1}{\partial z_1^i\partial z_1^j}\ll \left(\frac{Y_1\Xi_1}{X_1}\right)^i\left(\frac{\Xi_1}{X_1y_1}\right)^j\ll 1.\label{decayofUhormander}
\end{align}

We recall that 
$$\Xi(\xi,qm_1,h)=\frac{(T+r\xi)^{1/3}}{(2\pi (qm_1)^2h)^{1/3}},~ \Xi_2=\frac{(T+r\xi_2)^{1/3} }{(2\pi q_2^2m_1^2h_2)^{1/3}},~ \Xi_2'=\frac{(T+r\xi_2')^{1/3} }{(2\pi q_2'^2m_1^2h_2')^{1/3}}$$
and we define 
\begin{equation}\label{notation3}
\begin{split}
\Xi_1(z_1,y_1)&:=\frac{(T+r XY_1z_1)^{1/3}}{(2\pi (QY_1y_1)^2h_1)^{1/3}}\\
M(z_1,y)&:=\frac{(T+r\xi_2)^{1/3} }{(2\pi q_2^2m_1^2h_2)^{1/3}}-\frac{(T+r XY_1z_1)^{1/3}}{(2\pi (QY_1y_1)^2h_1)^{1/3}}=\Xi_2-\Xi_1,\\
M'(z_1,y)&:=\frac{(T+r\xi_2')^{1/3} }{(2\pi q_2'^2m_1^2h_2')^{1/3}}-\frac{(T+r XY_1z_1)^{1/3}}{(2\pi (QY_1y_1)^2h_1)^{1/3}}=\Xi_2'-\Xi_1,\\
L_k&:=\sqrt{q_2m_1}M^k-\sqrt{q_2'm_1}M'^k, \text{ for }k\in \mathbb Q,\\
q_1&=\frac{\tilde{u_1}}{m_1\tilde v_1d}, ~~\text{ and }~~F=\frac{q_1 Q}{mh_1h_2h_2'}=\frac{\tilde u_1 Q}{m_1d\tilde v_1mh_1h_2h_2'}.\\
\end{split}
\end{equation}
With the above set of notations we can write \begin{align}\mathcal{I}=XY_1^2\int_{\mathbb{R}}\int_{\mathbb{R}}e(f(z_1,y_1))U(z_1,y_1)dz_1dy_1\label{defnofcI}\end{align}where the phase function is 
\begin{equation}\label{Iintegralphasefunction}
f(z_1,y_1)=2\sqrt{QY_1y_1/m}L_{3/2}-FY_1y_1
.\end{equation}
Then, we deduce that
\begin{lemma}\label{AnalysisofI}
Let \begin{align}
    Q_1=m_1\tilde v_1d\cdot \frac{mh_1h_2h_2'}{Q}\cdot \frac{m_1(q_2'-q_2)X_1^{3/2}}{\sqrt{m}},~~~~ Y_1=\frac{X_1Q}{N^{1/3}}.\label{sizeofQ1}
\end{align}
When $|q_2-q_2'|\gg \frac{QT^{\varepsilon}}{m_1M_0^{1/3}Y_1X_1}$, $\mathcal{I}$ (\ref{defnofcI}) is negligibly small unless $\tilde u_1\sim Q_1$. 
In that range, we have $\Xi_2'-\Xi_2\sim \frac{(q_2-q_2')X_1}{Q/m_1}$ and 
\begin{equation}\label{mathcalI}
\begin{split}
\mathcal{I}\asymp  & X \cdot \frac{Q^2(u_2-u_2')^{-1}}{m_1d\tilde v_2M_0^{1/3}N^{1/3}}\cdot e\bigg(\frac{h_1h_2h_2'q_2q_2'm_1(\Xi_2'-\Xi_2)^3}{\tilde u_1(m_1\tilde v_1d)^{-1}(q_2-q_2')}\bigg)+O\left(\frac{XQ^2|q_2-q_2'|^{-2}}{M_0^{2/3}X_1^2m_1^2}\right).
\end{split}
\end{equation}
 \end{lemma}

\begin{proof}
   We will use Lemma \ref{Hormandertheorem} and Lemma \ref{Hormanderlemma } to prove this lemma, and we will frequently use the notations defined in (\ref{notation3}).  We note that the main term in (\ref{Hormanderlemmaequation}) is independent of the choice of $\lambda$. So we will choose it at the end. We also note that in the notation of Lemma \ref{Hormanderlemma }, the smooth function $U(z_1,y_1)$ is supported in a compact set $K\subset [Y_1^{-1}, 2Y_1^{-1}]^2$ and  by (\ref{decayofUhormander}), satisfies $\partial^i\partial^jU(z_1,y_1)\ll 1$. Then, we will calculate the first-order partial derivatives of $f$. 
    \begin{equation}\label{twovarfirstderivative}
\begin{split}
\frac{\partial f}{\partial y_1}&=\sqrt{\frac{QY_1}{my_1}}(L_{3/2}+2L_{1/2}\times\Xi_1)-FY_1,\\
\frac{\partial f}{\partial z_1}&=-\sqrt{\frac{QY_1y_1}{m}}L_{1/2}\times\frac{rXY_1\Xi_1}{(T+rXY_1z_1)}.
\end{split}
\end{equation}
where $L_k$s and $F$ are as defined in (\ref{notation3}).

Let $S\subset K$, where $\frac{\partial f}{\partial z_1}\gg T^{\varepsilon}$. In any subset of $S$, the integral is negligibly small by Lemma \ref{Hormandertheorem}. Now, in $S^c$, \begin{align}\frac{\partial f}{\partial z_1}\ll T^{\varepsilon}\implies |L_{1/2}|\ll T^{\varepsilon}\frac{T\sqrt{m}}{rXY_1\Xi_1\sqrt{Q}}\sim T^{\varepsilon}\frac{\sqrt{m}}{X_1\sqrt{Q}}\label{L1/2bound} .\end{align}
Multiplying $L_{1/2}$ (\ref{notation3}) by $\frac{1}{m_1}\sqrt{q_2m_1}M^{1/2}+\sqrt{q_2'm_1}M'^{1/2}\sim \frac{\sqrt{QX_1}}{m_1}$, and by (\ref{L1/2bound}), we get 
\begin{align*}
    |q_2M-q_2'M'|=|(q_2-q_2')(\Xi_2-\Xi_1)-q_2'(\Xi_2'-\Xi_2)|\ll \frac{T^{\varepsilon}\sqrt{m}}{X_1\sqrt{Q}}\cdot\frac{\sqrt{QX_1}}{m_1}\sim \frac{QT^{\varepsilon}}{m_1M_0^{1/3}}.
\end{align*}
And we already had $X_1\leq\Xi_2-\Xi_1\leq 2X_1$. Hence, 
\begin{align}
    \left|\frac{(q_2-q_2')X_1}{Q/m_1}-(\Xi_2'-\Xi_2)\right|\ll \frac{T^{\varepsilon}}{M_0^{1/3}}.\label{Xi_2q_2relation}
\end{align}
We emphasize on the fact that when $|q_2-q_2'|\gg \frac{QT^{\varepsilon}}{m_1M_0^{1/3}X_1}$, $\Xi_2'-\Xi_2\sim  \frac{(q_2-q_2')X_1}{Q/m_1}$. With this condition, we will analyze $\frac{\partial f}{\partial y}$ (\ref{twovarfirstderivative}). Following the notations in (\ref{notation3}), we derive that 
\begin{align}
    \sqrt{\frac{QY_1}{my_1}}(L_{3/2}+2L_{1/2}\Xi_1)
    =& \sqrt{\frac{QY_1}{my_1}}\left[(\Xi_2+\Xi_1)L_{1/2}-(\Xi_2'-\Xi_2)\sqrt{q_2m_1}M'^{1/2}\right].\nonumber
\end{align}
By (\ref{L1/2bound}), $\sqrt{\frac{QY_1}{my_1}}(\Xi_2+\Xi_1)L_{1/2}\ll T^{\varepsilon}$. Hence, in $S^c$, the integral is negligibly small unless 
\begin{align}
    \frac{\partial f}{\partial y}\ll T^{\varepsilon}\implies \left|\frac{Y_1 Q\sqrt{X_1}}{\sqrt m}\cdot(\Xi_2'-\Xi_2)+FY_1\right|\ll T^{\varepsilon}.\label{partialybound}
\end{align}
Here we have used $y_1\sim Y_1^{-1},~ q_2\sim Q/m_1$ and $M'\sim X_1$. Once we recall the definition of $F$ (\ref{notation3}) and $Q_1$ (\ref{sizeofQ1}), (\ref{partialybound}) along with (\ref{Xi_2q_2relation}) implies that the integral is negligibly small unless 
$$F\sim \frac{m_1(q_2'-q_2)X_1^{3/2}}{\sqrt{m}} \iff \tilde u_1\sim Q_1,$$
provided $|q_2-q_2'|\gg \frac{QN^{2\varepsilon}}{m_1M_0^{1/3}Y_1X_1}$. For the rest of the proof, we will assume this condition. 

 When $\tilde u_1\sim Q_1$, we have a unique stationary point $(z_0,y_0)$ such that $f'(z_0,y_0)=0$. We have 
\begin{align}
&\frac{\partial f}{\partial z_1}=0\iff L_{1/2}=0\iff q_2M=q_2'M' \iff \Xi_1(z_0)=\frac{q_2\Xi_2-q_2'\Xi_2'}{q_2-q_2'},\nonumber\\
&\frac{\partial f}{\partial y_1}=0\iff 	\sqrt{\frac{QY_1}{my_0}}L_{3/2}=FY_1\iff \sqrt{y_0}=\sqrt{\frac{Q}{mY_1}}F^{-1}L_{3/2}.\label{Xi_1}
\end{align} From (\ref{Xi_1}), we explicitly calculate $y_0$:  
\begin{equation}\label{M3/2}
\begin{split}
y_0&=\frac{Qm(h_1h_2h_2')^2}{q_1^2Q^2Y_1}\times\frac{q_2q_2'm_1(\Xi_2'-\Xi_2)^{3}}{(q_2-q_2')}.
\end{split}
\end{equation}
From (\ref{Iintegralphasefunction}), (\ref{Xi_1}), (\ref{M3/2}) and following the definitions of $F$ and $q_1$ (\ref{notation3}),  we derive  
\begin{equation*}\label{key}
\begin{split}
f(z_0,y_0)%&=2\sqrt{Qy_0/m}L_{3/2}-Fy_0=Fy_0\\
&=FY_1y_0=\frac{h_1h_2h_2'q_2q_2'm_1(\Xi_2'-\Xi_2)^3}{q_1(q_2-q_2')}.
\end{split}
\end{equation*}
Once we calculate the second-order partial derivatives of $f$ from (\ref{twovarfirstderivative}), we derive the determinant of the Hessian matrix 
$$\det(H_f)=\frac{-QY_1}{12my_1}\cdot  \frac{r^2X^2Y_1^2}{(T+rXY_1z_1)^2}\cdot[\Xi_1^2L_{3/2}L_{-1/2}+L_{1/2}\Xi_1(L_{3/2}+33L_{1/2}\Xi_1-12L_{-1/2}\Xi_1^2)],$$
where the $L_k$s are as defined in (\ref{notation3}). By (\ref{Xi_1}), we note that at the stationary point $(z_0,y_0)$, $L_{1/2}=0$ and at $(z_0,y_0)$,  
$$L_{3/2}L_{-1/2}=-\frac{\sqrt{q_2q_2'm_1}(\Xi_2'-\Xi_2)^{3/2}}{(q_2-q_2')^{1/2}}\cdot \frac{\sqrt{m_1}(q_2-q_2')^{3/2}}{\sqrt{q_2q_2'}(\Xi_2'-\Xi_2)^{1/2}}=-m_1(q_2-q_2')(\Xi_2'-\Xi_2). $$
With these observations and (\ref{Xi_2q_2relation}), we derive 
\begin{align*}
    \det(H_f(z_0,y_0))\sim \frac{M_0^{2/3}N^{2/3}Y_1^4}{Q^4}m_1^2(q_2-q_2')^2.
\end{align*}
Now, the main term of (\ref{mathcalI}) follows from (\ref{Hormanderlemmaequation}) of Lemma \ref{Hormanderlemma }.

Now, we choose our $\lambda$ (see \ref{Hormanderlemmaequation}) to be 
$$\lambda=\frac{M_0^{1/3}N^{1/3}Y_1^2}{Q^2}m_1|q_2-q_2'|\sim \frac{M_0^{1/3}Y_1X_1m_1|q_2-q_2'|}{Q}$$
and we have  $|\partial_1^{i}\partial_2^{j}U|\ll1$ by (\ref{decayofUhormander}). Thus, we get the error term, which completes the proof of our lemma.
\end{proof}
\begin{remark}
    The case of the zero frequency ($\tilde u_1=0$) is evaluated in Section \ref{u_1zero}. The case of $|q_2-q_2'|\ll \frac{QT^{\varepsilon}}{m_1M_0^{1/3}Y_1X_1}$ is considered in Section \ref{q_2small} and the contribution of the error term in (\ref{mathcalI}) is evaluated in Section \ref{errorterm}. Thus, in the rest of the work we will consider the main term of (\ref{mathcalI}) and only in the case of $\tilde u_1\neq 0$ and $|q_2-q_2'|\gg \frac{QT^{\varepsilon}}{m_1M_0^{1/3}Y_1X_1}$ and $X_1\gg \frac{T^{\varepsilon}}{M_0^{1/3}}$. 
\end{remark}
\subsection{Preparation for the $\xi_2$ integral}
When $X_1\gg \frac{T^{\varepsilon}}{M_0^{1/3}}$, after the application of the Poisson summation formula (Lemma \ref{afterpoissonofq1}) we can rewrite $\Delta$ (\ref{18.7}) as
\begin{align}
     \Delta\ll &\sup_{||\alpha||_2=1}\sup_{\frac{T^{\varepsilon}}{M_0^{1/3}}\ll X_1\ll \frac{N^{1/3}}{Q}}Q\sqrt{M_0}\Bigg(\sum_{m_1\ll Q}\sum_{d\ll M_1}\frac{1}{d}\underset{(\tilde v_1,\tilde v_2)=1}{\sum_{\tilde v_1,\tilde v_2|d^{\infty}}}\sum_{h_1\sim H_1}\sum_{\tilde{m}\sim M_1/d}\int_{\xi_2\sim X}\nonumber
        \\ &\sum_{h_2\sim H_1}\sum_{u_2\sim \frac{Q}{d\tilde v_2m_1}}\alpha_2\sum_{h_2'\sim H_1}\sum_{u_2'\sim \frac{Q}{d\tilde v_2m_1}} \frac{Q}{m_1\tilde v_1dmh_1h_2h_2'}\sum_{\tilde u_1\in \bz}|\cC|\int_{\xi_2'\sim X}\bar\alpha_2' \cdot \cI\Bigg)^{1/2}.\label{Deltaafterpoissononq1}
\end{align}
Now, let $|q_2-q_2'|\gg \frac{QT^{\varepsilon}}{m_1M_0^{1/3}Y_1X_1}$. Once we replace $\cI$ with its main term (\ref{mathcalI}) and inside the square root of (\ref{Deltaafterpoissononq1}), we apply Cauchy's inequality on all but the $\xi_2'$ integral, we get 
\begin{equation}
    \begin{split}
        \Delta\ll &\sup_{||\alpha||_2=1}\sup_{\frac{T^{\varepsilon}}{M_0^{1/3}}\ll X_1\ll \frac{N^{1/3}}{Q}}Q\sqrt{XM_0}\cdot S_3^{1/4}S_4^{1/4}.
    \end{split}
\end{equation}
where 
\begin{align}
    S_3=&\sum_{m_1\ll Q}\sum_{d\ll M_1}\frac{1}{d}\underset{(\tilde v_1,\tilde v_2)=1}{\sum_{\tilde v_1,\tilde v_2|d^{\infty}}}\int\limits_{\xi_2\sim X}\sum_{h_2\sim H_1}\sum_{u_2\sim \frac{Q}{d\tilde v_2m_1}}|\alpha_2|^2 \sum_{h_1\sim H_1}\sum_{\tilde{m}\sim\frac{M_1}{d}}\sum_{h_2'\sim H_1} \nonumber
        \\&\times \frac{Q}{m_1\tilde v_1dmh_1h_2h_2'}\cdot  \frac{Q^2}{m_1\tilde v_2dM_0^{1/3}N^{1/3}}\sum_{u_2'\sim \frac{Q}{d\tilde v_2m_1}}|u_2-u_2'|^{-1}\sum_{\tilde u_1\sim Q_1}|\cC|\cdot \label{S3sum}
\end{align}
and 
\begin{align}
    S_4=& \sum_{m_1\ll Q}\sum_{d\ll M_1}\frac{1}{d}\underset{(\tilde v_1,\tilde v_2)=1}{\sum_{\tilde v_1,\tilde v_2|d^{\infty}}}\sum_{h_1\sim H_1}\sum_{\tilde{m}\sim\frac{M_1}{d}}\sum_{h_2\sim H_1}\sum_{h_2' \sim H_1}\sum_{u_2'\sim \frac{Q}{d\tilde v_2m_1}}\nonumber 
        \\&\times \frac{Q}{m_1\tilde v_1dmh_1h_2h_2'}\cdot  \frac{Q^2}{m_1\tilde v_2dM_0^{1/3}N^{1/3}}\sum_{u_2\sim \frac{Q}{d\tilde v_2m_1}}|u_2-u_2'|^{-1}\sum_{\tilde u_1\sim Q_1}|\cC|\cdot S_{\xi},\label{S4sum}
\end{align}
 where 
 \begin{align}
     S_{\xi}=\int\limits_{\xi_2\sim X}\Bigg|\int\limits_{\xi_2'\sim X}\alpha_2'e\bigg(\frac{h_1h_2h_2'q_2q_2'm_1(\Xi_2'-\Xi_2)^3}{\tilde u_1(m_1\tilde v_1d)^{-1}(q_2-q_2')}\bigg)V\left(\frac{\Xi_2'-\Xi_2}{\frac{(q_2-q_2')X_1}{Q/m_1}}\right)\Bigg|^2.\label{S_xi}
 \end{align}
 By (\ref{sizeofQ1}), we get that $Q_1\ll m_1\tilde v_1dX_1^{3/2}\sqrt{M_1}H_1^3$. We also recall $H_1=\frac{QT}{N}$ and $M_1=\frac{Q^2X_1}{M_0^{2/3}}$. Then, evaluating the character sum $\cC$ by (\ref{charactersumgenericcase}), we can bound $S_3$ (\ref{S3sum}) by 
\begin{align}
    S_3=&\sum_{m_1\ll Q}\sum_{d\ll M_1}\underset{(\tilde v_1,\tilde v_2)=1}{\sum_{\tilde v_1,\tilde v_2|d^{\infty}}}\int\limits_{\xi_2\sim X}\sum_{h_2\sim H_1}\sum_{u_2\sim \frac{Q}{d\tilde v_2m_1}}|\alpha_2|^2 \sum_{h_2'\sim H_1} \sum_{h_1\sim H_1}\sum_{\tilde{m}\sim\frac{M_1}{d}}\nonumber
        \\&\times T^{\varepsilon}\frac{\sqrt{d_m}\cdot Q^3}{M_0^{1/3}N^{1/3}}\left(\frac{\sqrt{d}X_1^{3/2}\sqrt{H_1}\hat h_1}{m_1\tilde v_2}+\frac{1}{m_1^2\tilde v_1\tilde v_2\sqrt{M_1H_1d}}\right)\nonumber\\
        \ll&  \frac{T^{\varepsilon}\cdot Q^3}{M_0^{1/3}N^{1/3}}\left(X_1^{3/2}H_1^{5/2}M_1+H_1^{3/2}M_1^{1/2}\right)\ll  T^{\varepsilon}\left(\frac{Q^{15/2}X_1^{5/2}T^{5/2}}{M_0N^{17/6}}+\frac{Q^{11/2}T^{3/2}X_1^{1/2}}{M_0^{2/3}N^{11/6}}\right).\label{S3bound}
\end{align}
Here, we have used (\ref{vectorl2norm}) and 
\begin{align}
    \sum_{\tilde m\sim \frac{M_1}{d}}d_m=\sum_{\tilde m\sim \frac{M_1}{d}}(\tilde m, d^{\infty})\ll \frac{M_1^{1+\varepsilon}}{d}\text{ and }\sum_{h_1\sim H_1}\hat h_1=\sum_{h_1\sim H_1}(h_1, m^*)\ll H_1^{1+\varepsilon}.\label{hath1sum}
\end{align}
\subsection{$\xi_2$ integral}
In $S_{\xi}$ (\ref{S_xi}), we open up the absolute square and denote the $\xi_2$ integral by $\cI_2$
\begin{align*}
    \cI_2:=\int\limits_{\xi_2\sim X}e\bigg(\frac{[(\Xi_2'-\Xi_2)^3-(\Xi_2''-\Xi_2)^3]}{\tilde u_1(m_1\tilde v_1dh_1h_2h_2'q_2q_2'm_1)^{-1}(q_2-q_2')}\bigg)V\left(\frac{\Xi_2'-\Xi_2}{\frac{(q_2-q_2')X_1}{Q/m_1}}\right)V\left(\frac{\Xi_2''-\Xi_2}{\frac{(q_2-q_2')X_1}{Q/m_1}}\right)d\xi_2.
\end{align*}
Let us make a change of variable 
$\xi_2\mapsto X\cdot Y_2z_2$, where $Y_2:=\frac{m_1(q_2-q_2')X_1}{N^{1/3}}\ll \frac{QX_1}{N^{1/3}}$. Hence, $\Xi_2=\frac{(T+rXY_2z_2)^{1/3}}{(2\pi q_2^2m_1^2h_2)^{1/3}}$. 
With this change, we note that the smooth function 
$$U(z_2)=V(Y_2z_2)V\left(\frac{\Xi_2''-\Xi_2}{N^{1/3}Y_2/Q}\right)V\left(\frac{\Xi_2''-\Xi_2}{N^{1/3}Y_2/Q}\right)$$
satisfies the decay relation $\partial^jU(z_2)\ll 1$. Then, we can rewrite the $\cI_2$ integral as 
\begin{align*}
    \cI_2=XY_2\int\limits_{z_2}e\bigg(\frac{[(\Xi_2'-\Xi_2)^3-(\Xi_2''-\Xi_2)^3]}{\tilde u_1(m_1\tilde v_1dh_1h_2h_2'q_2q_2'm_1)^{-1}(q_2-q_2')}\bigg)U(z_2)dz_2
\end{align*}
If we denote the phase function by $f(z_2)$, we get 
\begin{align*}
    f'(z_2)=& \frac{m_1^2\tilde v_1dh_1h_2h_2'q_2q_2'(\Xi_2'-\Xi_2'')}{\tilde u_1(q_2-q_2')} \cdot \frac{rXY_2\Xi_2}{(T+rXY_2z_2)}\cdot [2\Xi_2-(\Xi_2'+\Xi_2'')]. 
\end{align*}
As $\Xi_2'-\Xi_2\sim \frac{X_1(q_2-q_2')}{Q/m_1}$ and $\Xi_2''-\Xi_2\sim \frac{X_1(q_2-q_2')}{Q/m_1}$, we get $2\Xi_2-(\Xi_2'+\Xi_2'')\sim -\frac{X_1(q_2-q_2')}{Q/m_1}$ in the support of $U(z_2)$. We also recall $\tilde u_1\sim Q_1\sim m_1^2v_1h_1h_2h_2'\sqrt{m}(q_2'-q_2)X_1^{3/2}Q^{-1}$.  Hence, 
$$|f'(z_2)|\sim   |\Xi_2'-\Xi_2''|\cdot \frac{Q\cdot q_2q_2'\cdot Y_2\cdot \frac{N^{1/3}}{Q} }{\sqrt{m}|q_2-q_2'|^2X_1^{3/2}} \cdot \frac{|q_2-q_2'|X_1}{Q/m_1}\sim M_0^{1/3}\cdot |\Xi_2'-\Xi_2''|.$$
Hence, by repeated integration by parts (Lemma \ref{repeatedint}), the integral is negligibly small unless 
\begin{align}
    |\Xi_2'-\Xi_2''|\ll & \frac{T^{\varepsilon}}{M_0^{1/3}} 
\iff |\xi_2'-\xi_2''|\ll XT^{\varepsilon}\cdot \frac{Q}{M_0^{1/3}N^{1/3}} \label{Xi2'X-2''range}
\end{align}
By Cauchy's inequality and symmetry of the two factors, we can bound $S_\xi$ (\ref{S_xi}) by 
\begin{align*}
  S_{\xi}=\int\limits_{\xi_2'\sim X}\alpha_2'\int\limits_{\xi_2''\sim X}\bar \alpha_2''\cdot \cI_2\ll \int\limits_{\xi_2'\sim X}|\alpha_2'|^2   \int\limits_{\xi_2''\sim X}|\cI_2|\ll \int\limits_{\xi_2'\sim X}|\alpha_2'|^2 \cdot \frac{X^2X_1Q^2}{M_0^{1/3}N^{2/3}}.
\end{align*}
In (\ref{S4sum}), if we replace $|S_{\xi}|$ with this bound, from (\ref{S3bound}), we have 
\begin{align}
    |S_4|\ll |S_3|\cdot \frac{X^2X_1Q^2}{M_0^{1/3}N^{2/3}}\ll \frac{X^2X_1Q^2T^{\varepsilon}}{M_0^{1/3}N^{2/3}}\left(\frac{Q^{15/2}X_1^{5/2}T^{5/2}}{N^{17/6}M_0}+\frac{Q^{11/2}X_1^{1/2}T^{3/2}}{M_0^{2/3}N^{11/6}}\right).\label{S4bound} 
\end{align}
Thus, using (\ref{S3bound}) and (\ref{S4bound}), in this case we can bound $\Delta$ by 
\begin{align*}
    \Delta&\ll \sup_{||\alpha||_2=1}\sup_{\frac{T^{\varepsilon}}{M_0^{1/3}}\ll X_1\ll \frac{N^{1/3}}{Q}}Q\sqrt{XM_0}\cdot S_3^{1/4}S_4^{1/4}
    \ll\frac{T^{\varepsilon}Q^{15/4}XT^{5/4}}{M_0^{1/12}N^{13/12}}+\frac{T^{\varepsilon}Q^{15/4}XT^{3/4}M_0^{1/12}}{N^{11/12}}.
\end{align*}
Finally, in this case, we can bound $\mathfrak{M}$ (Lemma \ref{Largesievemain}) by 
\begin{align}
    \mathfrak{M} \ll&\frac{N^{4/3}}{T}\sup_{Q\ll Q_0}\frac{1}{Q^2}\sup_{M_0\ll \frac{N^{2+\varepsilon}}{Q_0^3}}M_0^{1/3}\times \Delta\ll  \frac{N^{4/3}}{T}\cdot\frac{X^{3/2}T^{3/4}}{N^{1/12}}\left(1+\frac{N^{1/4}}{X^{1/4}T^{1/4}}\right)\nonumber\\
    %\ll& N^{3/4}T^{1/2}X^{3/2}
    \ll& N^{5/4}T^{-1/4}X^{3/2}.\label{mathfrakMgeneric} 
\end{align}    
\subsection{Non-generic cases}
\subsubsection{$X_1\ll \frac{T^{\varepsilon}}{M_0^{1/3}}$}\label{X_1small}
When $m\neq 0$ and $X_1\ll \frac{T^{\varepsilon}}{M_0^{1/3}}$, we have $m\ll \frac{Q^2T^{\varepsilon}}{M_0}$. Hence, $S_{m_2}\ll \frac{M_0}{m_1^2}\sum_{m\ll \frac{Q^2T^{\varepsilon}}{M_0} }|\mathfrak{C}_0|\ll \frac{dQ^2T^{\varepsilon}}{m_1^2}. $
From the dyadic relation $\Xi_2-\Xi_1\asymp X_1$, we get 
$|\xi_2-B\xi_1|\ll \frac{XX_1Q}{N^{1/3}}$
where $B$ is depends on $q_1,q_2,h_1,h_2$ with value bounded between absolute constants ($B\sim 1$). Then, $\Delta$ (\ref{DeltaAfterfirstpoisson}) is bounded  by 
\begin{align*}\Delta
\ll& T^{\varepsilon}\sup_{X_1\ll \frac{T^{\varepsilon}}{M_0^{1/3}}}\sum_{m_1\ll Q}m_1^2\cdot \frac{XX_1Q}{N^{1/3}}\cdot \frac{Q}{m_1}\cdot \frac{QT}{N}\cdot Q^2
\ll \frac{T^{\varepsilon}{TXQ^5}}{N^{4/3}M_0^{1/3}}.
\end{align*}
Hence, 
\begin{align}
    \mathfrak{M} 
    \ll& \frac{N^{4/3+\varepsilon}}{T}\frac{TXQ_0^3}{N^{4/3}}\ll T^{\varepsilon} XQ_0^3\asymp N^{3/2+\varepsilon}X^{5/2}T^{-3/2}.\label{mathfrakMX_1small}
\end{align}
\subsubsection{$\tilde u_1=0$}\label{u_1zero}
In this case, there is no stationary phase in the $\cI$ (\ref{defnofcI}). From (\ref{partialybound}), we note that the integral $\cI$ is negligibly small unless 
$|\Xi_2-\Xi_2'|\ll \frac{T^{\varepsilon}}{Y_1M_0^{1/3}}\iff |\xi_2'-B\xi_2|\ll \frac{T^{\varepsilon}XQ}{N^{1/3}Y_1M_0^{1/3}},$
for some $B\sim 1$ depending on $q_2,q_2',h_2,h_2'$. From this relation and (\ref{Xi_2q_2relation}), we get that $|u_2-u_2'|\ll \frac{QT^{\varepsilon}}{m_1d\tilde v_2Y_1X_1M_0^{1/3}}.$ Then, we can bound $\cI$ (\ref{mathcalI}) by $XY_1^2$. 
Using (\ref{charsumu1zero}) to bound $\cC$ (along with the condition $\tilde h_2=\tilde h_2'$) and following the steps of the generic case, we can bound $\Delta$ (\ref{Deltaafterpoissononq1}) by 
\begin{align*}
    \Delta\ll T^{\varepsilon} \frac{XQ^{3}}{N^{1/6}}\left(M_0^{1/12}+\frac{T^{1/2}}{N^{1/3}}\right)\ll \frac{T^{\varepsilon}XQ^{3}M_0^{1/12}}{N^{1/6}}.
\end{align*}
Hence, 
\begin{align}
    \mathfrak{M} \ll&\frac{N^{4/3}}{T}\sup_{Q\ll Q_0}\frac{1}{Q^2}\sup_{M_0\ll \frac{N^{2+\varepsilon}}{Q_0^3}}M_0^{1/3}\times \Delta
    \ll \frac{N^{4/3+\varepsilon}Q_0X}{TN^{1/6}}\cdot \frac{N^{5/6}}{Q_0^{5/4}}
   \ll N^{15/8+\varepsilon}X^{7/8}T^{-7/8}.\label{MathfrakMzero}
\end{align}
\subsubsection{$|q_2-q_2'|\ll \frac{QT^{\varepsilon}}{m_1M_0^{1/3}Y_1X_1}$}\label{q_2small}
In this case, the phase function of $\cI$ (\ref{defnofcI}) does not have a stationary point. But from (\ref{partialybound}), we have 
$ |\xi_2'-B\xi_2|\ll \frac{XQ}{N^{1/3}Y_1M_0^{1/3}}$ and $\tilde u_1\ll \frac{mh_1h_2h_2'm_1\tilde v_1dT^{\varepsilon}}{Y_1Q} $, otherwise $\cI$ is negligibly small. Then, using (\ref{charsumu3smallu1nonzero}) to bound $\cC$, we can bound (\ref{Deltaafterpoissononq1}) 
\begin{align}
    \Delta\ll &\frac{T^{\varepsilon}Q^{21/4}XT^{9/4}}{N^{9/4}M_0^{1/4}}\nonumber\\
    \text{ and }~~\mathfrak{M}\ll& N^{3/4}X^{5/2}T^{-1/4}\label{MathfrakMsmallq}
\end{align}
where the second term is always smaller than the first term. 
\subsubsection{Error term}\label{errorterm}
In this case, we consider the error term of (\ref{mathcalI}), i.e., $\cI\ll \frac{XQ^2|q_2-q_2'|^{-2}}{M_0^{2/3}X_1^2m_1^2}.$ The dyadic condition $\Xi_2'-\Xi_2\sim \frac{(q_2-q_2')X_1}{Q/m_1}$ implies $\xi_2'-B\xi_2\sim \frac{XX_1m_1(q_2-q_2')}{N^{1/3}}$ where $B$ depends on $q_2,q_2',h_2,h_2'$ and $B\sim 1$. Then, the rest of the analysis is very similar to the analysis leading to (\ref{S3bound}) and we achieve 
\begin{align}
    \Delta\ll & QX \left(\frac{Q^{3}T^{5/4}}{M_0^{1/6}N^{7/6}}+\frac{Q^{11/4}T^{3/4}M_0^{1/12}}{N^{11/12}}\right)\nonumber
    ~~~\text{and}~~~\mathfrak{M}\ll N^{5/4+\varepsilon}X^{7/4}T^{-1/2}+N^{1+\varepsilon}X^{3/4}\label{mathfrakMerror}
\end{align}

\section{Proof of Theorem \ref{maintheorem}}\label{maintheoremproof}
We collect all the bound on $\mathfrak{M}$ from (\ref{mathfrakmdia}), (\ref{mathfrakMgeneric}), (\ref{mathfrakMX_1small}), (\ref{MathfrakMzero}), (\ref{MathfrakMsmallq}) and (\ref{mathfrakMerror}) and we get 
\begin{align}
    \mathfrak{M}\ll T^{\varepsilon}\Bigg(&\frac{N^{3/2}\sqrt{T}}{\sqrt{X}}+\frac{N^{5/4}X^{3/2}}{T^{1/4}}+\frac{N^{3/2}X^{5/2}}{T^{3/2}}+\frac{N^{15/8}X^{7/8}}{T^{7/8}}+\frac{N^{3/4}X^{5/2}}{T^{1/4}}+\frac{N^{5/4}X^{7/4}}{T^{1/2}}\Bigg).\nonumber
\end{align}
We equate the first (diagonal) and the second bound (generic off-diagonal) to get 
$$\frac{N^{3/2}\sqrt{T}}{\sqrt{X}}=N^{5/4}T^{-1/4}X^{3/2}\iff X=N^{1/8}T^{3/8}.$$
As $N\ll T^{3/2+\varepsilon}$, all the other bounds are smaller than the first bound. This choice of $X$ also satisfies the condition $Q_0\gg T/X$ because $N\gg T^{15/11+\varepsilon}$. Hence, 
\begin{align}
    \mathfrak{M}\ll N^{23/16+\varepsilon}T^{5/16}
\end{align}
and \begin{align*}
    M_F(T)\ll &\sup_{T^{15/11+\varepsilon}\ll N\ll T^{3/2+\varepsilon}}\frac{1}{NX}\cdot X\cdot \frac{T}{X}\cdot \mathfrak{M}+T^{15/11+\varepsilon}
    \ll T^{3/2-3/32+\varepsilon},
\end{align*}
which concludes the proof of Theorem \ref{maintheorem}. 
\section*{Acknowledgment}
The author would like to thank Prof. Ritabrata Munshi for his continuous guidance and support throughout the course of this work. The author is grateful to Prof. Yongxiao Lin for his comment on Corollary \ref{cor2}. The author would like to thank Mayukh Dasaratharaman, Sumit Kumar, Kummari Mallesham, and Prahlad Sharma for many helpful discussions. The author is also grateful to the Indian Statistical Institute, Kolkata, for providing an excellent research atmosphere. 


\begin{thebibliography}{Xyz1234}
\bibitem[ALM22]{ALM} K. Aggarwal, W. H. Leung, R. Munshi, \emph{Short second moment bound and Subconvexity for $\mathrm{GL}(3)$ $L$-functions}, preprint, arXiv:2206.06517v1 [math.NT].
\bibitem[At49]{At}F. V. Atkinson, \emph {The mean-value of the Riemann zeta function},  Acta Math. \textbf{81} (1949), 353-376. 
MR0031963.  DOI: 10.1007/BF02395027.
MR4203038.  DOI : 10.24033/asens.2451.
\bibitem[BKY13]{BKY}V. Blomer, R. Khan, M. Young, \emph {Distribution Of Mass Of Holomorphic Cusp Forms}, Duke Math. J. \textbf{162} (2013), no.14, 2609-2644. MR3127809. DOI:10.1215/00127094-2380967  
\bibitem[CFKRS05]{CFKRS}J. B. Conrey, D. W. Farmer, J. P. Keating, M. O. Rubinstein, N. C. Snaith, \emph{Integral Moments of $L$-Functions}, Proc. London Math. Soc.(3) \textbf{91} (2005), no.1,  33-104. MR2149530. DOI:10.1112/S0024611504015175.     
\bibitem[CGh84]{CGh}J. B. Conrey, A. Ghosh, \emph{On mean values of the zeta-function}, Mathematika \textbf{31} (1984), no.1,  159-161. 
MR0762188 . DOI:10.1112/S0025579300010767
\bibitem[CGo01]{CGo}J. B. Conrey, S. M. Gonek, \emph{High moments of the Riemann zeta-function}, Duke Math. J. \textbf{107} (2001), no.3, 577-604. MR1828303. DOI:10.1215/S0012-7094-01-10737-0 
\bibitem[DLY24]{DLY} A. Dasgupta, W. H. Leung, M. P. Young, \emph{The Second Moment of the $GL_3$ standard $L$-function on the critical line}, preprint, arXiv:2407.06962v1 [math.NT]
\bibitem[DGH03]{DGH}A. Diaconu, D. Goldfeld, and J. Hoffstein,  \emph{Multiple Dirichlet series and moments of zeta and $L$-functions}, Compositio Math. \textbf{139} (2003), no.3, 297-360. MR2041614.  DOI:10.1023/B:COMP.0000018137.38458.68
\bibitem[Gf]{Gf}D. Goldfeld, {\it Automorphic forms and L-functions for the group $\mathrm{GL}(n,\mathbb{R})$}. With an appendix by Kevin A. Broughan, Cambridge Studies in Advanced Mathematics, 99. Cambridge University Press, Cambridge, 2006. %ISBN: 978-0-521-83771-2; 0-521-83771-5
\bibitem[GL06]{GL}D. Goldfeld, X. Li, \emph{Voronoi formulas on $\mathrm{GL}(n)$}, Int. Math. Res. Not. 2006, Art. ID 86295, 25 pp. MR2233713. DOI: 10.1155/IMRN/2006/86295 
\bibitem[Go82]{Go}A. Good, \emph{The square mean of Dirichlet series associated with cusp forms}, Mathematika \textbf{29}  (1982), no.2,  278-295. MR0696884.  DOI:10.1112/S0025579300012377
\bibitem[Hb78]{Hb1}Heath-Brown, \emph{The twelfth power moment of the Riemann zeta-function}, Quart. J. Math. Oxford Ser.(2) \textbf{29} (1978), no.116, 443-462. MR0517737. DOI:10.1093/qmath/29.4.443
\bibitem[Hb79]{Hb2}Heath-Brown, \emph{The fourth power moment of the Riemann zeta- function}, Proc. London Math. Soc.(3) \textbf{38} (1979), 385-422. 
MR0532980.  DOI:10.1112/plms/s3-38.3.385 
\bibitem[HL16]{HL}G. H. Hardy, J. E. Littlewood, \emph{Contributions to the theory of the Riemann zeta-function and the theory of the distribution of primes}, Acta Math. \textbf{41} (1916), no.1,  119-196. 
MR1555148. DOI: 10.1007/BF02422942 
\bibitem[HMQ23]{HMQ} R. Holowinsky, R. Munshi b, Z. Qi, \emph{Beyond the Weyl barrier for $\mathrm{GL}(2)$ exponential sums}, Adv. of Math. \textbf{426} (2023). DOI:10.1016/j.aim.2023.109099
\bibitem[Ho]{Ho} L. Hormander,{\it The analysis of linear partial differential operators. I. Distribution theory and Fourier analysis.}, Grundlehren der mathematischen Wissenschaften [Fundamental Principles of Mathematical Sciences], vol. 256, Springer-Verlag, Berlin, 1983. %ISBN: 3-540-12104-8
\bibitem[Hu21]{Hu} B. Huang, \emph{On the Rankin-Selberg problem}, Math. Ann. \textbf{381} (2021), no.3-4, 1217-1251, MR4333413. DOI:10.1007/s00208-021-02186-7 
\bibitem[IK04]{IK}H. Iwaniec, E. Kowalski, {\it  Analytic number theory}, American Mathematical Society Colloquium Publications,53. American Mathematical Society, Providence, RI, 2004.   %ISBN: 0-8218-3633-1.
\bibitem[In27]{In}A. E. Ingham, \emph{Mean-Value Theorems in the Theory of the Riemann Zeta-Function}, Proc. London Math. Soc. (2) \textbf{27} (1927), no.4,  273-300. MR1575391. DOI:10.1112/plms/s2-27.1.273. 
\bibitem[KS00]{KS} J. P. Keating, N. C. Snaith, \emph{Random matrix theory and $\zeta(1/2+it)$}, Comm. Math. Phys. \textbf{214} (2000),no.1, 57-89. 
MR1794265. DOI: 10.1007/s002200000261 
\bibitem[Li09]{Li1} X. Li, \emph{The Central Value of the Rankin-Selberg	L-functions}, Geom. Funct. Anal. \textbf{18}(2009), no.5, 1660 – 1695. MR2481739. DOI:10.1007/s00039-008-0692-5
\bibitem[Li11]{Li2}  X. Li, \emph{Bounds for $\mathrm{GL}(3) \times  \mathrm{GL}(2)$ $L$-functions and $\mathrm{GL}(3)$ $L$-functions}, Ann. of Math. (2) \textbf{173} (2011), no.1, 301-336. MR2753605. DOI:10.4007/annals.2011.173.1.8  
\bibitem[LNQ22]{LNQ}Y. Lin, R. Nunes, Z. Qi, \emph{Strong Subconvexity for Self-Dual $\mathrm{GL}(3)$ $L$-Functions}, Int. Math. Res. Not. 2022, rnac153. DOI: 10.1093/imrn/rnac153  
\bibitem[Mo]{Mo} Y. Motohashi, {\it Spectral theory of the Riemann zeta-function}, Cambridge Tracts in Mathematics, 127. Cambridge University Press, Cambridge, 1997. MR1489236  
\bibitem[MS07]{MS}A. Mukhopadhyay, K. Srinivas, \emph{A zero density estimate for the Selberg class}, Int. J. Number Theory \textbf{3} (2007), no. 2, 263-273. 
MR2333620. DOI:10.1142/S1793042107000894  
\bibitem[MSY18]{MSY}M. McKee, H. Sun, Y.Ye, \emph{Improved subconvexity bounds for $\mathrm{GL}(2)\times \mathrm{GL}(3)$ and $\mathrm{GL}(3)$ $L$-functions by weighted stationary phase}, Trans. Amer. Math. Soc. \textbf{370} (2018), no.5,  3745-3769. MR3766865.  DOI:10.1090/tran/7159
\bibitem[Mu15]{Mu1}R. Munshi, \emph{The circle method and bounds for $L$-functions—III: $t$-aspect subconvexity for $\mathrm{GL}(3)$ $L$-functions}, J. Amer. Math. Soc.  \textbf{28} (2015), no.4,  913-938. MR3369905. DOI: 10.1090/jams/843
\bibitem[Mu22]{Mu2}R. Munshi, \emph{Subconvexity for $\mathrm{GL}(3)\times \mathrm{GL}(2)$ $L$-functions in $t$-aspect}, J. Eur. Math. Soc. \textbf{24} (2022), no.5, 1543–1566. MR4404783. 10.4171/JEMS/1131.   
\bibitem[Ne21]{Ne}P. D. Nelson, \emph{Bounds for standard L-functions}, preprint, arXiv:2109.15230  [math.NT]
\bibitem[Nu17]{Nu}R. M. Nunes, \emph{On the subconvexity estimate for self-dual $\mathrm{GL}(3)$ $L$-functions in the $t$-aspect}, arXiv:1703.04424v1 [math.NT]
\bibitem[Ra39]{Ra}R. A. Rankin, \emph{Contributions to the theory of Ramanujan’s function $\tau(n)$ and similar arithmetical functions. I. The zeros of the function$\sum_{n=1}^{\infty}\tau(n)/n^s$ on the line $\Re s = 13/2$. II. The order of the Fourier coefficients of integral modular forms}, Proc. Camb. Philos. Soc. \textbf{35} (1939), 351-372. MR0000411 
\bibitem[Se40]{Se}A. Selberg, \emph{Bemerkungen über eine Dirichletsche Reihe, die mit der Theorie der Modulformen nahe verbunden ist.} (German), Arch. Math. Nat. \textbf{43} (1940), 47-50. 
MR0002626.  
\bibitem[Yo11]{Yo}M. Young, \emph{The second moment of $\mathrm{GL}(3) \times \mathrm{GL}(2)$ $L$-functions integrated}, Adv. Math.\textbf{226} (2011), no.4, 3550-3578. MR2764898. DOI:10.1016/j.aim.2010.10.021.   
\bibitem[YZ13]{YZ} Y. Ye, D. Zhang, \emph{Zero density for automorphic $L$-functions}, J. Number Theory \textbf{133} (2013), 3877-3901. MR3084304. DOI:10.1016/j.jnt.2013.05.012.
\end{thebibliography}
\end{document}